\documentclass{amsart}
\usepackage{amsmath,amssymb,amscd,enumerate}
\usepackage[all]{xy}
\usepackage[pdftex]{graphicx,color}
\usepackage{url}

\usepackage[colorlinks,linkcolor=blue,citecolor=blue,pdfstartview=FitH]{hyperref}

\newcommand{\co}{\colon \thinspace}
\newcommand{\onto}{\twoheadrightarrow}

\newcommand\call{\mathcal{L}}
\newcommand\caln{\mathcal{N}}
\newcommand\cals{\mathcal{S}}
\newcommand\rk{\mathrm{rk}\,}
\newcommand\rg{\mathrm{rg}\,}
\newcommand\lcap{\wedge_{\mathcal{L}}}
\newcommand\essint{\cap_{\mathit{ess}}}

\newcommand\acurve{\mathfrak{a}}
\newcommand\bcurve{\mathfrak{b}}

\theoremstyle{definition}
\newtheorem{para}{}[section]
\newtheorem{dfn}[para]{Definition}

\theoremstyle{plain}
\newtheorem{thm}[para]{Theorem}
\newtheorem{lemma}[para]{Lemma}
\newtheorem{cor}[para]{Corollary}
\newtheorem*{thm1}{Theorem}

\newtheorem{prop}[para]{Proposition}

\theoremstyle{remark}

\newtheorem*{remark}{Remark}

\begin{document}

\title{Explicit rank bounds for cyclic covers}
\author{Jason DeBlois}
\address{Department of Mathematics, University of Pittsburgh}
\email{jdeblois@pitt.edu}
\thanks{This paper was originally written during a period of partial NSF support}

\begin{abstract}  For a closed, orientable hyperbolic 3-manifold $M$ and a homomorphism $\phi \co \pi_1(M) \onto \mathbb{Z}$, that is not induced by a fibration $M\to S^1$, we give a lower bound on the ranks of the subgroups $\phi^{-1}(n\mathbb{Z})$, $n\in\mathbb{N}$, that is linear in $n$.  The key new ingredient is the following result: if $M$ is a closed, orientable hyperbolic $3$-manifold and $S$ is a connected, two-sided incompressible surface of genus $g$ that is not a fiber or semi-fiber, then a reduced homotopy in $(M,S)$ has length at most $14g-12$.
\end{abstract}
\maketitle

The \textit{rank} of a group $G$, $\rk G$, is the minimal cardinality of a generating set.  This paper gives lower bounds on rank of $\pi_1$ among cyclic covers of certain $3$-manifolds:

\begin{thm}\label{explicit}  For a closed, orientable hyperbolic $3$-manifold $M$, a homomorphism $\phi\co\pi_1 M\onto\mathbb{Z}$ and $n\geq 2$, let $\Gamma_n = \phi^{-1}(n\mathbb{Z})$.  Let $\|\phi\|$ denote the Thurston norm of the cohomology class of $\phi$.  If $\phi$ is not induced by a fibration $M\to S^1$ then
$$ \rk \Gamma_n \geq \frac{n-1}{7\|\phi\|+2} $$\end{thm}

The \textit{Thurston norm} of the cohomology class of $\phi$ is defined as the minimum, taken over all surfaces $S$ embedded in $M$ representing the Poincar\'e dual of $\phi$, of $\sum_{i=1}^k \max\{-\chi(S_i),0\}$, where the $S_i$ are the components of $S$.  See \cite{Thurston_norm}.  Theorem \ref{explicit} immediately implies the following bound on \textit{rank gradient} of the pair $(\pi_1 M, \{\Gamma_n\})$, defined by Lackenby \cite{Lack_rank} as $\liminf_{n\to\infty} (\rk \Gamma_n-1)/n$.

\begin{cor}\label{closed posgrad}  For a closed, orientable hyperbolic $3$-manifold $M$, a homomorphism $\phi\co \pi_1 M \onto \mathbb{Z}$ and $n\geq 2$, let $\Gamma_n = \phi^{-1}(n\mathbb{Z})$.  Let $\|\phi\|$ denote the Thurston norm of the cohomology class of $\phi$.  If $\phi$ is not induced by a fibration $M \to S^1$ then:
$$  \rg(\pi_1 M, \{\Gamma_n\}) \geq 1/(7\|\phi\|+2) $$\end{cor}

If $\phi$ is induced by a fibration then $\rk \pi_1 M_n \leq 2g+1$ for every $n$, where $g$ is the genus of a connected fiber.  Hence $\rg (\pi_1 M, \{\Gamma_n\}) = 0$.  In the earlier paper \cite{DFV}, joint with Stefan Friedl and Stefano Vidussi, we proved a weaker analog of Corollary \ref{closed posgrad} for a broader class of $3$-manifolds: compact, orientable and connected with toroidal or empty boundary.  For such $M$ and $\phi\co\pi_1 M\onto\mathbb{Z}$, Theorem 1.1 there implies that $\rg (\pi_1 M, \{\Gamma_n\}) > 0$ if $\phi$ is not induced by a fibration.

The proof strategy of \cite[Theorem 1.1]{DFV} is to find a finite cover $p\co M'\to M$ inheriting a map $\phi'\co\pi_1M'\onto\mathbb{Z}$ so that $\rg(\pi_1 M',\{(\phi')^{-1}(n\mathbb{Z})\})>0$ for homological reasons, whence $\rg(\pi_1 M,\{\Gamma_n\})>0$ as well.  The ``virtually special'' machine of Wise et.~al. produces $p$, and controlling its degree seems out of reach at present.  Producing an explicit bound thus requires a different strategy.  Our approach here, outlined in Section \ref{trees}, instead follows that of Section 3 of \cite{DFV}.  We use:

\theoremstyle{plain}
\newtheorem*{Corollary_cylinders}{Corollary \ref{cylinders}}
\newcommand\cylinders{For a closed, orientable hyperbolic $3$-manifold $M$ and a connected, two-sided incompressible surface $S\subset M$ of genus $g$ that is not a fiber or semi-fiber, the $\pi_1 M$-action on the tree determined by $S$ is $(14g-12)$-acylindrical.}
\begin{Corollary_cylinders} \cylinders \end{Corollary_cylinders}

Combining this with an acylindrical accessibility theorem of R.~Weidmann \cite{Weidmann} immediately gives Theorem \ref{explicit}.  The action at issue above is described by Bass--Serre theory, see eg.~\cite{ScoWa}.  A connected surface $S \subset M$ is a \textit{semi-fiber} if it separates $M$ into a disjoint union of twisted $I$-bundles over the non-orientable surface double covered by $S$.  If $S$ is a semi-fiber, then there is a twofold cover $\widetilde{M}\to M$ such that $S$ lifts to a fiber of a fibration $\widetilde{M}\to S^1$.  It is necessary in Corollary \ref{cylinders} that $S$ not be a fiber or semi-fiber:  otherwise each element of $\pi_1 S <\pi_1 M$ fixes the entire tree, so the action is $k$-cylindrical for all $k\geq 0$. 

Corollary \ref{cylinders} in turn follows from Theorem \ref{bdd length} below, whose proof contains the main substantive work of the paper.  It is an extension of the so-called ``veg-o-matic'' argument which has seen prior use in the works of Cooper--Long \cite[\S 4]{CooLo}, Li \cite[\S 2]{TaoLi}, Walsh \cite{Walsh}, and  Boyer--Culler--Shalen--Zhang \cite[Theorem 5.4.1]{BCSZ}.

\newcommand\BddLength{For a closed, orientable hyperbolic $3$-manifold $M$ and a connected, two-sided incompressible surface $S\subset M$ of genus $g$ that is not a fiber or semi-fiber, a non-degenerate, reduced homotopy in $(M,S)$ has length at most $14g-12$.}
\newtheorem*{BddLengthThm}{Theorem \ref{bdd length}}
\begin{BddLengthThm}\BddLength\end{BddLengthThm}  

Above, a homotopy in $(M,S)$ is a map of pairs $H\co(K\times I,K\times\partial I)\to (M,S)$, for a topological space $K$.  It is \textit{reduced of length $k$} if it obtained by chaining together homotopies $H^1,\hdots,H^k$ such that $H^i$ is essential and $(H^i)^{-1}(S) = K\times\partial I$ for each $i$, and $H^{i+1}$ starts on the opposite side of $S$ from which $H^i$ ends for $i<k$.  (See also Definition \ref{BCSZ stuff}.)  An observation of Z.~Sela draws the connection between homotopies through $M$ of curves in $S$ and cylinders of the $\pi_1 M$-action on the tree of $S$.  In Section \ref{cylinders vs cylinders}, we reproduce this observation as Lemma \ref{chopped annulus}.  With Theorem \ref{bdd length}, it immediately implies Corollary \ref{cylinders}.

Section \ref{essential} gives some results on intersections of surfaces that we use in Section \ref{shorter cylinders} to prove Theorem \ref{bdd length}.  This argument has two main steps.  The first identifies a sequence $\Psi_1 \supset \Psi_2\supset\cdots$ of subsurfaces of $S$, of minimal complexity with the property that for each $k$, a reduced homotopy $H$ with length $k$ and target $(M,S)$ has $H_0$ homotopic into $\Psi_k$ in $S$.  The primary technical tool in this step is the characteristic submanifold of the manifold obtained by cutting $M$ along $S$.

The second step uses the fact that $M$ is hyperbolic and $S$ is not a fiber or semi-fiber to show that $\Psi_k$ is not homotopic into $\Psi_{k+2}$ in $S$ as long as $\Psi_k \neq \emptyset$.  Therefore eventually $\Psi_k = \emptyset$, and homotopies expire in finite time.

For various reasons, previous versions of this argument do not require accounting for solid torus components of the characteristic submanifold.  However, homotopies through $M$ of curves in $S$ may indeed pass through such solid tori.  The difficulty in extending the standard argument to accommodate this is that the time-$0$ map of a homotopy through such a component may not determine the time-$1$ map.  

We sidestep this issue, producing the $\Psi_k$ by adding judiciously chosen annuli to a sequence $\{\Phi_k\}$ of subsurfaces of $S$, identified in \cite{BCSZ}, that carry time-$0$ maps of  ``large'' homotopies (see Definition \ref{large'n'stuff}) with target $(M,S)$.   Indeed, many of the results of Sections \ref{essential} and \ref{shorter cylinders} rely on and directly extend work in \cite{BCSZ}.  We indicate when this is so and cross-reference precisely.

\subsection*{Acknowledgment}  Many thanks to Peter Shalen for explaining the argument of \cite{BCSZ} to me, and for helping me extend it to the present context.  Thanks also to Peter Scott and the anonymous referee for suggestions that have improved the paper.

\section{Proof of the main theorem} \label{trees}

The proof of Theorem \ref{explicit} closely follows the proof of Theorem 3.4 of \cite{DFV}.  We will sketch it below, at times referring to \cite{DFV} for details.  But first we recall the definition of an acylindrical action and reproduce an ``acylindrical accessibility'' theorem of R.~Weidmann. 

\begin{dfn}[\cite{Sela}]\label{acylindrical}  $\Gamma \times T \to T$ is \textit{$k$-acylindrical} if no $g \in \Gamma - \{1\}$ fixes a segment of length greater than $k$, and \textit{$k$-cylindrical} otherwise.  \end{dfn}

\begin{thm1}[Weidmann, \cite{Weidmann}]  Let $\Gamma$ be a non-cyclic, freely indecomposable, finitely generated group and $\Gamma \times T \to T$ a minimal $k$-acylindrical action.  Then $\Gamma \backslash T$ has at most $1+2k(\rk \Gamma - 1)$ vertices.  \end{thm1}

Assuming Corollary \ref{cylinders}, we now sketch the proof of Theorem \ref{explicit}.

\begin{proof}[Proof sketch, Theorem \ref{explicit}]  Let $M$ be a closed hyperbolic $3$-manifold and $\phi\co\pi_1 M\to \mathbb{Z}$ a homomorphism not induced by a fibration $M\to S^1$.  Standard arguments produce a closed, oriented surface $S$ embedded in $M$ that is \textit{dual} to $\phi$ in the sense that $\phi = p_*$ for the map $p\co M\to S^1 = [-1,1]/(-1\sim1)$, defined as follows: for a tubular neighborhood  $\mathcal{N} = S\times [-1,1]$ of $S$ in $M$ and $(x,t)\in S\times [-1,1]$, let $p(x,t) = t$, and let $p(x) = -1\sim1$ for each $x\in M-\mathcal{N}$.

There is a $\pi_1$-surjective map $q\co M\to G_0$, where $G_0$ is a graph with one vertex for each component of $M-(S\times (-1,1))$ and one edge for each component of $S\times [-1,1]$ (with the obvious attaching maps), such that $p$ factors through $q$.  If $\chi(G_0) < 0$ then for each $n\geq 2$, $\rk\pi_1 M_n \geq -n\chi(G_0)+1$.  This follows from the fact that $M_n$ $\pi_1$-surjects an $n$-fold cover of $G_0$, the motivating observation for Lemma 3.3 of \cite{DFV} (see also Lemma 2.6 there).

By the above the result holds if $\chi(G_0) < 0$, so we may assume $\chi(G_0) = 0$.  Assuming that $S$ has minimal complexity among all surfaces dual to $\phi$, it follows that $G_0$ has one vertex and one edge; i.e.~$S$ is connected and non-separating.  This assertion is proved in the final two paragraphs of the proof of Theorem 3.4 of \cite{DFV}.  Here the \textit{complexity} of $S = S_1\sqcup \hdots \sqcup S_k$, where each $S_i$ is connected, is defined as $\chi_- S = \sum_{i=1}^k \max\{-\chi(S_i),0\}$.  The Thurston norm $\|\phi\|$ of $\phi$ is by definition equal to $\chi_- S$ for $S$ dual to $\phi$ with minimal complexity.

$G_0$ is the underlying graph of a \textit{graph of spaces} decomposition of $M$ in the sense of \cite[p.~155]{ScoWa}, with vertex space $X = \overline{M-(S\times[-1,1])}$ and edge space $S\times[-1,1]$.  There is an associated action of $\pi_1 M$ on a tree $T$, without inversions, such that each vertex stabilizer is conjugate to $\pi_1(X)$ and each edge stabilizer to $\pi_1 S_0$ for some component $S_0$ of $S$.  See \cite[pp.~166--167]{ScoWa}, also \cite{Serre} and \cite{Tretkoff}.  This is what we call ``the action on the $\pi_1 M$-tree determined by $S$''.

We know that $S$ is not the fiber of a fibration $M\to S^1$ (if it were then $p$ would be homotopic to a fibration, contradicting our hypothesis on $\phi$), and since it is non-separating it is not a semi-fiber.  Corollary \ref{cylinders} therefore asserts that the $\pi_1 M$-action on $T$ is $(14g-12)$-acylindrical, where $g$ is the genus of $S$.  This property is inherited by each subgroup $\Gamma_n=\phi^{-1}(n\mathbb{Z})<\pi_1 M$.  By construction the graph $\Gamma_n\backslash T$ has $n$ vertices and edges, so the result follows directly from Weidmann's theorem upon noting that $\|\phi\| = 2g-2$.\end{proof}

\section{Cylinders and homotopies}\label{cylinders vs cylinders}

\noindent We reproduce example (iv) of Z.~Sela's introduction to \cite{Sela} below:\medskip

\indent\parbox[t]{335pt}{Let $S$ be an incompressible surface in a compact $3$-manifold $M$.  Let $M'$ denote the $3$-manifold obtained by cutting $M$ along $S$.  A homotopy $H$ [in $M$] between two closed curves in $S$ can be decomposed into essential homotopies in $M'$.  The number of these essential subhomotopies is called the \textit{length} of $H$.  An incompressible surface is called $k$-\textit{acylindrical} if no homotopy between closed curves in $S$ has length bigger than $k$.  To an incompressible surface $S$ in $M$ corresponds a splitting of $\pi_1M$.  The bound on the length of a homotopy between curves on $S$ corresponds exactly to the dual splitting being $(k+1)$-acylindrical.}

\medskip\noindent The purpose of this section is to expand on Sela's remarks, define his terms, and give a reasonably detailed sketch proof of the assertion of his final sentence in our case.

\begin{lemma}\label{chopped annulus}  Let $M$ be a closed, irreducible $3$-manifold and $S\subset M$ a closed, connected, two-sided incompressible surface.  For $k>1$, the action $\pi_1 M \times T \to T$ on the tree $T$  determined by $S$ is $k$-cylindrical if and only if there is a non-degenerate reduced homotopy $(S^1\times I,S^1\times\partial I)\to (M,S)$ of length $k$.  \end{lemma}

A surface $S$ as above is \textit{incompressible} if it is embedded in $M$ with $\pi_1$-injective inclusion map, and it is not a two-sphere that bounds a ball in $M$.  See eg.~\cite[Ch.~6]{Hempel}.  We prove Lemma \ref{chopped annulus} at the end of this section, but first note that combining it with Theorem \ref{bdd length} immediately yields:

\begin{cor}\label{cylinders}\cylinders\end{cor}

\begin{dfn}\label{basic stuff}  Let $X$ and $Y$ be topological spaces.  A \textit{homotopy with domain $X$ and target $Y$} is a map $H \co X \times I \to Y$.  The \textit{time-$t$ map of $H$}, $H_t \co X \to Y$, is defined by $H_t(x) = H(x,t)$.  For a map $f \co X \to Y$, a \textit{homotopy of $f$} is a homotopy $H$ with $H_0 =f$.  A map $g \co X \to Y$ is \textit{homotopic to $f$} if there is a homotopy $H$ of $f$ with $H_1 =g$.

Let $H^1,\hdots,H^n$ be homotopies with domain $X$ and target $Y$.  A homotopy $H$ with domain $X$ and target $Y$ is the \textit{composition of $H^1,\hdots,H^n$} if there exist numbers $0=t_0<t_1\cdots<t_n=1$ and monotone increasing linear homeomorphisms $\alpha_i\co[t_{i-1},t_i] \to [0,1]$ such that $H(x,t) = H^i(x,\alpha_i(t))$ for all $x\in X$ and $t\in[t_{i-1},t_i]$.

A path $\gamma\co I \to Y$ may be regarded as a homotopy with domain $X = \emptyset$.  We will denote the composition of paths $\gamma_1,\gamma_2,\hdots,\gamma_n$, as defined above, by $\gamma_1.\gamma_2.\cdots.\gamma_n$.

For $Z \subset Y$, we say $f\co X \to Y$ is \textit{homotopic into $Z$} if $f$ is homotopic to a map $g$ with $g(X) \subset Z$.  If $W \subset X$, a \textit{homotopy of $W$} is a homotopy of the inclusion map $W \to X$.  A map of pairs $f \co (X,W) \to (Y,Z)$ is \textit{essential} if $f$ is not homotopic through maps $(X,Y) \to (Z,W)$ to a map into $W$.  \end{dfn}

The definitions above are standard.  We have borrowed their precise formulations from \cite{BCSZ}.  This is also our source for  the definitions below that apply to $3$-manifolds.

\begin{dfn}\label{BCSZ stuff}  Let $M$ be a closed $3$-manifold and $S \subset M$ an embedded, transversely oriented surface.  A \textit{homotopy in $(M,S)$ with domain $K$} is a homotopy $H$ with domain $K$ and target $M$ such that $H(K \times \partial I) \subset S$.  It is \textit{non-degenerate} if $H_*(\pi_1 K)\neq\{1\}$, and \textit{basic} if $H^{-1}(S) = K \times \partial I$.

For $\epsilon \in \{+,-\}$, we say a basic homotopy \textit{starts} (or \textit{ends}) \textit{on the $\epsilon$-side} if $H(K\times [0,\delta]) \subset \caln_{\epsilon}$ (or, respectively, if $H(K\times [1-\delta,1]) \subset \caln_{\epsilon}$).  Here $\caln \cong S\times [-1,1]\subset M$ is a closed regular neighborhood of $S$, embedded so that $S = S\times \{0\}$ and the standard transverse orientation is preserved, and $\caln_+ = S\times [0,1]$, $\caln_- = S\times[-1,0]$.

We say that $X = M - (S\times (-1,1))$ is \textit{obtained by cutting $M$ along $S$}.  If $H$ is a basic homotopy in $(M,S)$ with domain $K$ then after straightening in $\caln$ and reparametrizing, the restriction of $H$ to $H^{-1}(X)$ determines a homotopy $H'$ in $(X,\partial X)$ with domain $K$.  We say $H$ is \textit{essential} if $H'$ is essential as a map of pairs $(K \times I, K \times \partial I) \to (X,\partial X)$; i.e. $\pi_1$-injective and not properly homotopic into $\partial X$.  

A homotopy $H$ in $(M,S)$ with domain $K$ is \textit{reduced with length $k$} if there exist basic essential homotopies $H^1,\hdots,H^k$ and $\epsilon_i \in \{+,-\}$ for $1\leq i\leq k$ such that $H$ is the composition of $H^1,\hdots,H^k$, and for each $i<k$ $H^i$ starts on the $\epsilon_i$- and ends on the $-\epsilon_{i+1}$-side, and $H^k$ starts on the $\epsilon_k$-side.  \end{dfn}

A connected, incompressible surface $S$ in a closed $3$-manifold $M$ determines a \textit{graph of spaces decomposition} of $M$ whose underlying graph $G$ has a single edge, corresponding to $S$, and (one or two) vertices corresponding to the components of the manifold $X$ obtained by cutting $M$ along $S$.  By Bass--Serre theory this determines an action of $\pi_1 M$ on a tree $T$, without inversions and with quotient graph $G$.  We will use the following basic consequence of this set-up.

\begin{lemma}\label{treecrawler}  Suppose a group $\Pi$ acts on a tree $T$, transitively on edges and without inversions.  Let $\{e_0,\hdots,e_k\}$ be a segment of $T$ of length $k+1$, so $e_i\neq e_{i-1}$ but $e_i$ and $e_{i-1}$ share an endpoint $v_i$ for each $i>0$, and let $\Lambda = \mathrm{Stab}_{\Pi}(e_0)$.\begin{description}
  \item[S]  If $G = \Pi\backslash T$ has two vertices let $\Gamma_- = \mathrm{Stab}_{\Pi}(v_0)$ and $\Gamma_+ = \mathrm{Stab}_{\Pi}(v_1)$, where $v_0\neq v_1$ is an endpoint of $e_0$ (recall from above that $v_1 = e_1\cap e_0$).  Then for each $i\in\{1,\hdots, k\}$ there exists $\gamma_i$, in $\Gamma_- - \Lambda$ for $i$ even and $\Gamma_+ - \Lambda$ for $i$ odd, such that for each $j\leq k$, $\delta_j = \gamma_1\gamma_2\cdots\gamma_j$ takes $e_0$ to $e_j$.
  \item[N] If $G$ has a single vertex let $\Gamma = \mathrm{Stab}_{\Pi}(v_0)$, let $\Lambda_+<\Gamma$ stabilize an edge $e'$ containing $v_0$ but not $\Gamma$-equivalent to $e_0$, and fix $\tau\in\Pi$ with $\tau(e')=e_0$.  Orient the edge of $G$ so that $e_0$ points toward $v_1$ in the inherited orientation on $T$, and for $0<i\leq k$ let $\epsilon_i = 1$ if then $e_i$ points from $v_i$ to $v_{i+1}$; $\epsilon_i = -1$ otherwise.  For each $i\in\{1,\hdots,k\}$ there exists $\gamma_i\in\Gamma$ so that for $1\leq j\leq k$,
  $$ \delta_j = \tau(\gamma_1 \tau^{\epsilon_1}) \cdots (\gamma_{j-1}\tau^{\epsilon_{j-1}}) $$
has the property that $\delta_j(v_0) = v_j$, and $\displaystyle{e_j = \left\{\begin{array}{ll} \delta_{j}\gamma_j(e_0) &\mathrm{if}\ \epsilon_j = 1. \\
 \delta_{j}\gamma_j(e') & \mathrm{if}\ \epsilon_j = -1. \end{array}\right.}$\\
Let $\epsilon_0=1$.  For $i\geq 1$ if $\epsilon_{i-1} \neq\epsilon_i$, then $\gamma_{i+1}$ is not in an edge stabilizer.
\end{description}\end{lemma}

\begin{proof}  In case \textbf{S} $T$ has two $\Pi$-orbits of vertices, and the stabilizer of any vertex $v$ acts transitively on the edges containing it.  This is because on a small neighborhood $U$ of $v$ in $T$, the projection to $\Pi\backslash T$ factors through an embedding of $\mathrm{Stab}_{\Pi}(v)\backslash U$.  This case, which we leave to the reader, is a straightforward induction argument.

With notation as described in case \textbf{N}, there are two $\Gamma$-orbits of edges of $T$ containing $v_0$: one pointing toward $v_0$ and one away.  In particular, $e'$ points toward $v_0$, and $\tau(v_0) = v_1$.  We therefore take $\delta_1 = \tau$.  Then $\delta_1^{-1}(e_1)$ contains $v_0$, so depending on orientation it is $\Gamma$-equivalent to one of $e'$ or $e_0$.  If $e_1$ points toward $v_1$ then $\gamma_1^{-1}\delta_1^{-1}(e_1) = e'$ for some $\gamma_1\in\Gamma$; otherwise there exists $\gamma_1\in\Gamma$ with $\gamma_1^{-1}\delta_1^{-1}(e_1) = e_0$.  This proves the base case of an induction argument.

For the inductive step we take $j>1$ and suppose we have identified $\delta_{j-1}$ and $\gamma_{j-1}$ satisfying the required properties.  If $\epsilon_{j-1} =1$ it follows that $\gamma_{j-1}^{-1}\delta_{j-1}^{-1}(v_j) = v_1$, so $\tau^{-1}\gamma_{j-1}^{-1}\delta_{j-1}^{-1}(v_j) = v_0$; whereas if $\epsilon_{j-1} = -1$ we have $\tau\gamma_{j-1}^{-1}\delta_{j-1}^{-1}(v_j) = v_0$.  Therefore $\delta_j(v_0)=v_j$.  Since $\delta_j^{-1}(e_j)$ thus contains $v_0$, arguing as in the base case we identify $\gamma_j\in\Gamma$ so that $\delta_j\gamma_j$ takes $e_0$ or $e'$ (depending on $\epsilon_j$) to $e_j$.

For the lemma's final assertion, we note that $\epsilon_{i-1}\neq\epsilon_i$ implies that $e_{i-1}$ and $e_i$ either both point toward or both away from $v_i$, so $\delta_i^{-1}(e_{i-1})$ and $\delta_i^{-1}(e_i)$ are distinct $\Gamma$-translates.  A definition-chase shows the translating element is $\gamma_i$.\end{proof}

To the graph of spaces decomposition determined by an incompressible surface $S$, there corresponds ``graph of groups'' decomposition of the fundamental group of $M$ with underlying graph $G$.  We record this in the standard lemma below, a paraphrase of \cite[p.~155]{ScoWa}.

Here for a closed path $\gamma$ based at a point $x\in M$ we will also denote its based homotopy class in $\pi_1(M,x)$ by $\gamma$, letting context determine the proper interpretation.  And we let $\alpha.\beta$ denote the composition of paths $\alpha$ and $\beta$, defined as in Definition \ref{basic stuff}.

\begin{lemma}\label{Bass-Serre}  For a closed $3$-manifold $M$ and a connected, transversely oriented incompressible surface $S$ with closed regular neighborhood $\caln\cong S\times [-1,1]\subset M$, let $S_{\pm} = S\times\{\pm 1\}$ and $X = M - (S\times(-1,1))$.  Fix $x\in S$, take $x_{\pm} = (x,\pm 1)\in S_{\pm}$, let $\Lambda = \pi_1(S_-,x_-)$, and let $\alpha\co t\mapsto (x,2t-1)$ join $x_-$ to $x_+$ in $\caln$.\begin{description}
  \item[S]  If $S$ is separating then $\pi_1 M$ is a \mbox{\rm free product with amalgamation}:
  $$ \pi_1(M,x_-) \cong \Gamma_-*_{\Lambda}\Gamma_+ \doteq \left\langle \Gamma_-,\Gamma_+\,|\, \lambda = \alpha.\phi_*(\lambda).\bar{\alpha},\ \lambda\in \Lambda \right\rangle$$
  Here $\Gamma_- = \pi_1 (X_-,x_-)$, where $X_-$ is the component of $X$ with $S_- =\partial X_-$, and $\Gamma_+ = \{\alpha.\gamma.\bar{\alpha}\,|\,\gamma\in\pi_1(X_+,x_+)\}$ for $X_+$ with $S_+ = \partial X_+$.
  \item[N]  If $S$ is non-separating then $\pi_1 M$ is an \mbox{\rm HNN extension} of $\Gamma = \pi_1(X,x_-)$:\begin{align}\label{HNN}
  \pi_1 (M,x_-)\cong \Gamma*_{\Lambda} \doteq \left\langle\,\Gamma,\tau\,|\,\tau^{-1}\lambda\tau = \bar{\beta}.\phi_*(\lambda).\beta,\ \lambda\in \Lambda \right\rangle\end{align}
Here $\tau\in\pi_1(M,x_-)$ is the pointed homotopy class of $\alpha.\beta$ for some fixed arc $\beta$ in $X$ joining $x_+$ to $x_-$.\end{description}
In each case above, $\phi\co S_-\to S_+$ takes $(x,-1)$ to $(x,1)$ for all $x\in S$, so $\phi_*$ is an isomorphism $\Lambda\to\pi_1(S_+,x_+)$.\end{lemma}

\begin{proof}[Proof of Lemma \ref{chopped annulus}]  First suppose there is a non-degenerate, reduced homotopy $H\co (S^1\times I,S^1\times\partial I)\to (M,S)$ of length $k$.  Writing $H$ as a composition of essential basic homotopies $H^1,\hdots,H^k$, we may assume without loss of generality that $H^{-1}(S) = \bigsqcup_{i=0}^k S^{-1}\times\{i/k\}$ and for each $i>0$, $H^i$ linearly reparametrizes $H|_{S^1\times[(i-1)/k,i/k]}$.  We may further assume that $H$ is \textit{vertical} with respect to a closed regular neighborhood $\caln\cong S\times[-1,1]$ of $S$ in $M$, by which we mean that each $H^i$ is obtained from its restriction to $(H^i)^{-1}(X)$ by collaring, where $X = M - S\times (-1,1)$.

Let $\hat{H}^i\co S^1\times I\to X$ be obtained by reparametrizing the restriction of $H^i$ to the preimage of $X$.  Fix a base point $x\in S$ and for each $i$ fix a path $\rho^i$ in $S$ from $x$ to $H(1,i/k)$.  Taking $S_{\pm} = S\times\{\pm 1\}$, let $x_{\pm} = (x,\pm 1)\in S_{\pm}$, and let $\rho^i_{\pm 1}$ be the path parallel to $\rho^i$ in $S_{\pm}$.

Assume for now that $S$ is separating.  Then each $H^i$ starts and ends on the same side of $S$, so since $H$ is reduced the $H^i$ alternate sides.  We will assume that $H^i$ starts and ends on the $+$-side for odd $i$ and the $-$-side for even $i$ (the argument in the other case is completely analogous).  Thus $\hat{H}^i$ maps into $X_{-}$ for $i$ odd and $X_{+}$ for $i$ even, where $X_{\pm}$ is the component of $X$ with $S_{\pm} = \partial X_{\pm}$.  For $1\leq i\leq k$ define:
$$\gamma_i = \left\{\begin{array}{cl} 
  \rho^{i-1}_-.(t\mapsto\hat{H}^i(1,t)).\bar{\rho}^i_- & i\ \mbox{even} \\
  \alpha.\rho^{i-1}_+.(t\mapsto\hat{H}^i(1,t)).\bar{\rho}^i_+.\bar{\alpha} & i\ \mbox{odd}
\end{array}\right.$$
Here $\alpha$ is as described as in Lemma \ref{Bass-Serre}.  For $\Gamma_{\pm}$ as described there it follows by construction that $\gamma_i\in \Gamma_+$, $i$ odd, and $\gamma_i\in\Gamma_+$ otherwise.

We claim that for all $i$, $\gamma_i\notin\Lambda$. If $\gamma_i\in\Lambda$ then after a homotopy of $H$, $H^i|_{\{1\}\times I}\subset S$, so there is a map $(D,\partial D)\to(M,S)$ factoring through $H$ for a disk $D$.  Since $S$ is incompressible $\partial D$ bounds a disk $D'\subset S$.  The sphere theorem and irreducibility of $M$ imply that $\pi_2(M) = 0$, so $D\cup D'\to M$ extends over a ball.  It follows that $H$ is not essential, contradicting our hypotheses.

For each $i\in\{0,\hdots,k\}$ one obtains a loop in $M$ based at $H(1,0)$ by applying $H$ to the concatenation of the straight-line path in $S^1\times I$ joining $(1,0)$ to $(1,i/k)$ with the loop around $S^1\times\{i/k\}$, followed by the straight-line path back to $(1,0)$.  After connecting the base point $x_-$ to $H(1,0)$ using $\rho^0_-$ and a vertical arc, these loops all evidently represent the same element $g$ of $\pi_1(M,x_-)$.  A short induction argument shows that $\delta_i^{-1}g\delta_i\in \Lambda$ for all $i$, where $\delta_i = \gamma_1\cdots\gamma_i$.

As we remarked above Lemma \ref{treecrawler}, by Bass--Serre theory $S$ determines an action $\pi_1 M\times T\to T$ on a tree $T$, without inversions and with quotient graph $G$.  Under this action, the stabilizer of each edge is a conjugate of the edge group $\Lambda$ of $G$, and the stabilizer of each vertex is conjugate to a vertex group of $G$: in this case one of $\Gamma_{\pm}$.  See \cite[pp.~166--167]{ScoWa}.

Let $e_0$ be the edge of $T$ stabilized by $\Lambda$, and let $v_0$ and $v_1$ be the endpoints of $e$ stabilized respectively by $\Gamma_-$ and $\Gamma_+$.  Then $g$ is in $\Lambda$ and by construction also in $\delta_i\Lambda\delta_i^{-1}$ for each $i$, stabilizing $e_i=\delta_i(e)$.  These determine a path in $T$ since $\gamma_i = \delta_{i-1}^{-1}\delta_i$ is in one of $\Gamma_+$ or $\Gamma_-$ for each $i$.  This path has length $k+1$ because $\gamma_i\notin\Lambda$, so $e_i\neq e_{i-1}$, for any $i$.

The separating case of the ``if'' direction of Lemma \ref{chopped annulus} is established.  Note that the elements $\gamma_i$ and $\delta_i$ above match the descriptions in case \textbf{S} of Lemma \ref{treecrawler}.  

Suppose now that $S$ is non-separating, so $X$ is connected with two boundary components $S_{\pm}$.  Given a non-degenerate homotopy $H$ of length $k$, decomposed into $H^1,\hdots,H^k$ as previously, there are four possibilities for each $H^i$.  If $H^i$ starts and ends on the $-$-side we define $\gamma_i$ as for $H^i$, $i$ even, in the separating case, and if it starts and ends on the $+$-side we define as for $i$ odd.  Otherwise:
$$ \gamma_i= \left\{\begin{array}{cl}
  \rho^{i-1}_-.(t\mapsto\hat{H}^i(1,t)).\bar{\rho}^i_+.\beta & H^i\ \mbox{starts on the}\ +,\ \mbox{ends on the}\ -\ \mbox{side} \\
  \alpha.\rho^{i-1}_+.(t\mapsto\hat{H}^i(1,t)).\bar{\rho}^i_-.\bar{\beta}.\bar{\alpha} & H^i\ \mbox{starts on the}\ -,\ \mbox{ends on the}\ +\ \mbox{side} \end{array}\right.$$
Here $\beta$ is as described in case \textbf{N} of Lemma \ref{Bass-Serre}.  If $H^i$ starts and ends on the same side then arguing as in the separating case shows $\gamma_i$ is not in an edge stabilizer.  We produce a path in $T$ by a process similar to the separating case, using words $\delta_j$ which in this case match the description in case \textbf{N} of Lemma \ref{treecrawler} (for $\tau$ as described in Lemma \ref{Bass-Serre}).  The details of this case track those of the parallel case of the reverse implication, described below.

We now address the reverse implication of the Lemma, proving that a nontrivial element $g$ stabilizing a length-$(k+1)$ segment in $T$ gives rise to a length-$k$ reduced homotopy in $(M,S)$.  The idea of the proof is to use the description of Lemma \ref{treecrawler} to reverse-engineer the construction above.  We leave the separating case of this construction to the reader (it is simpler), and move directly to the case that $S$ is non-separating.  The four different boundary behaviors of basic homotopies in this case correspond to the possible orientations on edges meeting at a vertex.

To make this precise let us fix some notation.  For $\Gamma$, $\Lambda$ and $\tau$ defined as in case \textbf{N} of Lemma \ref{Bass-Serre}, $\Gamma$ stabilizes a vertex $v_0$ of $T$ and $\Lambda<\Gamma$ stabilizes an edge $e_0$ containing $v_0$.   It further follows from (\ref{HNN}) above that $e' \doteq \tau^{-1}(e_0)$ contains $v_0$ since $\Lambda_+\doteq\tau^{-1}\Lambda\tau < \Gamma$.

Suppose now that the $\pi_1 M$-action is $k$-cylindrical, so there exists $g\in\pi_1 M-\{1\}$ fixing a segment of length at least $k+1$.  By transitivity, upon replacing $g$ by a conjugate we may assume $v_0$ is the segment's initial vertex.  Each edge containing $v_0$ is a $\Gamma$-translate of exactly one of $e_0$ or $e'$, since $X$ has two boundary components.  Thus conjugating $g$ further in $\Gamma$, we may assume the segment's initial edge is either $e'$ or $e_0$.  If it is $e_0$ we apply Lemma \ref{treecrawler}; if $e'$ we exchange $\Lambda$ and $\Lambda_+$, replace $\tau$ by $\tau^{-1}$ rename $e'$ to $e_0$ and vice-versa, then apply case \textbf{N} of Lemma \ref{treecrawler}.

For each $j\geq0$, since $g$ stabilizes $e_j$ the lemma implies that $g = (\delta_j\gamma_j)\lambda_j(\delta_j\gamma_j)^{-1}$ for some $\lambda_j$, which is in $\Lambda$ if $\epsilon_j = 1$ and $\Lambda_+$ otherwise.  Since $\delta_j = \delta_{j-1}(\gamma_{j-1}\tau^{\epsilon_{j-1}})$, comparing the resulting descriptions of $g$ at $e_{j-1}$ and $e_j$, for $j>0$, yields:\begin{align}\label{comparison}
  \tau^{-\epsilon_{j-1}}\lambda_{j-1}\tau^{\epsilon_{j-1}} = \gamma_j\lambda_j\gamma_j^{-1}
\end{align}
For each $j$ such that $\epsilon_j = 1$, fix a closed curve $\mathfrak{c}_j$ on $S_-$ through $x_-$ that represents $\lambda_j$.  If $\epsilon_j = -1$ then since $\Lambda_+ = \tau^{-1}\Lambda\tau$, $\lambda_j \in\Lambda_+$ is of the form $\tau^{-1}\lambda_j^{(0)}\tau = \bar{\beta}.\phi_*(\lambda_j^{(0)}).\beta$ for some $\lambda_j^{(0)}\in\Lambda$.  In this case let $\mathfrak{c}_j$ be a closed curve on $S_+$ that represents $\phi_*(\lambda_j^{(0)})\in\pi_1(S_+,x_+)$.

For each $j>0$, equation (\ref{comparison}) above determines a homotopy in $X$ from $\phi(\mathfrak{c}_{j-1})$ (if $\epsilon_{j-1} = 1$) or $\phi^{-1}(\mathfrak{c}_{j-1})$ (if $\epsilon_{j-1} = -1$) to $\mathfrak{c}_j$.  One produces from this a basic homotopy $H^j$ in $(M,S)$ by adjoining product collars in the obvious way.  By construction, $H^{j+1}$ starts on the opposite side of $S$ from $H^j$ for each $j<k$.  To show that the composition of $H^1,\hdots,H^k$ is reduced of length $k$, it remains only to show that each $H^j$ is essential.

This is clear when $H^j$ starts and ends on opposite components of $\partial X$, so let us consider a case where it does not.  If $\epsilon_{j-1} = 1$ and $\epsilon_j = -1$ then $\phi(\mathfrak{c}_{j-1}) = H^j(S^1\times\{0\})$ and $\mathfrak{c}_j = H^j(S^1\times\{1\})$ each lie in $S_+$.  Equation (\ref{comparison}) becomes
$$ \bar{\beta}.\phi_*(\lambda_{j-1}).\beta = \gamma_j (\bar{\beta}.\phi_*(\lambda_j^{(0)}).\beta)\gamma_j^{-1}, $$
and $H^j$ is a concatenation of four homotopies: the free homotopy from $\phi_*(\lambda_{j-1})$ to $\bar{\beta}.\phi_*(\lambda_{j-1}).\beta$, the pointed homotopy between left and right sides of (\ref{comparison}), the free homotopy from $\gamma_j(\bar{\beta}.\phi_*(\lambda_j^{(0)}).\beta)\gamma_j$, and finally the free homotopy to $\phi_*(\lambda_j^{(0)})$.

If there were a proper homotopy of $H^j$ into $S_+$ it would follow that $\gamma_j\in\Lambda_+$, contradicting the final assertion of Lemma \ref{treecrawler}.  The case $\epsilon_{j-1} = -1$, $\epsilon_j = 1$ is similar.\end{proof}

\section{Essential surfaces and essential intersections}\label{essential}

Now we shift gears to extend the theory of ``essential intersection'' for subsurfaces of a $2$-manifold that is introduced in \cite[\S 4]{BCSZ}.  There it is remarked that this notion ``has appeared implicitly in much of the literature on the characteristic submanifold of a Haken manifold''.  The results of \cite{BCSZ} are proved for \textit{large} subsurfaces (see below); we must allow annular components as well.  Many results extend directly to this context using similar proof strategies, but some require important caveats.

We will work in the PL category throughout the next two sections.  In particular, a \textit{polyhedron} is a topological space that admits the structure of a simplicial complex.  It is well known that the class of such spaces includes surfaces and $3$-manifolds.  We also use ``annulus" interchangeably with ``cylinder'' to refer to $S^1\times I$.

\begin{dfn}\label{large'n'stuff}  Let $S$ be an orientable surface with no $2$-sphere components.  If $K$ is a polyhedron, we will say a map $f\co K \to S$ is \textit{$\pi_1$-injective} if on each component $K_0$ the induced map on $\pi_1 K_0$ is injective, and \textit{large} if this map has nonabelian image.  

If $A \subset S$ is a subsurface, we will say $A$ is \textit{incompressible} if no component of $A$ is a disk and the inclusion map $A\to S$ is $\pi_1$-injective.  A component $A_0$ of an incompressible subsurface $A$ is \textit{redundant} if its inclusion map is homotopic in $S$ into another component of $A$.  We say $A \subset S$ is \textit{irredundant} if it is incompressible and has no redundant components.

If $A$ is a compact orientable surface, we will refer to the union of the components of $A$ with negative Euler characteristic as the \textit{large part} $A_{\call}$, and to the union of the core circles of the remaining annular components as the \textit{small part} $A_{\cals}$ of $A$.  (Note that $A_{\call}\cup A_{\cals}$ is properly contained in $A$.)\end{dfn}

\begin{remark}  If $A$ and $B$ are orientable surfaces and $h\co A \to B$ is a $\pi_1$-injective map, then $h(A_{\call})\subset B_{\call}$.  \end{remark}

The kind of argument we will use in this section is illustrated by a sketch proof for the following assertion: if $A$ is an incompressible subsurface of an orientable surface $S$ with no $2$-sphere components then each redundant component of $A$ is homeomorphic to an annulus.  

Suppose $A_0$ is such a component, whose inclusion map is homotopic in $S$ into another component $A_1$.  We may assume $A_1$ lies in the interior $\mathrm{int}\, S$ of $S$, after pushing off the boundary.  Choosing a basepoint in $A_1$, let $\tilde{S} \to \mathrm{int}\, S$ be the cover corresponding to $\pi_1 A_1$.  The inclusion $A_1 \to S$ lifts to an embedding to a subsurface $\tilde{A}_1 \subset \tilde{S}$ that carries $\pi_1 \tilde{S}$.  Therefore each component of $\tilde{S} - \mathrm{int}\,\tilde{A}_1$ is homeomorphic to a half-open annulus.  Since the inclusion map of $A_0$ is homotopic into $A_1$, it lifts to an embedding in $\tilde{S}$.  The inclusion map $A_0\to S$ is $\pi_1$-injective by hypothesis, so its lift is too, and the lift's image does not intersect $\tilde{A}_1$.  The latter fact implies that its image is contained in a half-open annulus, so $A_0$ is an orientable surface with cyclic fundamental group; hence an annulus.

The lemma below extends Lemma 4.1 of \cite{BCSZ}.

\begin{lemma}\label{homotopy to iso}  Suppose $A$ and $B$ are irredundant subsurfaces of a compact, orientable surface $S$ with no $2$-sphere or torus components, and $A$ is homotopic into $B$.  \begin{enumerate} 
\item\label{isotopic into}  $A$ is isotopic in $S$ to a subsurface of $B$.
\item\label{homeo to iso}  If $B$ is homeomorphic to an irredundant subsurface of $A$, then $A$ and $B$ are isotopic subsurfaces of $S$.
\item\label{isotopic to equal} If $B$ is homotopic into $A$, then $A$ and $B$ are isotopic subsurfaces of $S$.  
\end{enumerate}  \end{lemma}

\begin{proof}  We follow the outline of the proof of \cite[Lemma 4.1]{BCSZ}; as there, assume without loss of generality that $S$ is connected.  If $S$ is an annulus, then any irredundant subsurface of $S$ is also an annulus, and the conclusions of the lemma follow quickly.  We thus assume below that $S$ has negative Euler characteristic.

We first prove (\ref{isotopic into}).  We initially consider only $A_{\call} \cup A_{\cals}$; that is, the disjoint union of the large part $A_{\call}$ of $A$ and the $1$-submanifold $A_{\cals}$ consisting of the cores of the annular components.  Analogous to \cite[Lemma 4.1]{BCSZ}, we choose this object within its isotopy class so that $\partial A_{\call} \cup A_{\cals}$ meets $\partial B$ transversely in the minimal number of points possible; and, among all intersection-minimizing representatives, to minimize the number of components of $\partial A_{\call}\cup A_{\cals}$ not contained in $B$.

Given a component $A_0$ of $A_{\call}$ that is homotopic into a component $B_0$ of $B$, the proof of Lemma 4.1(1) in \cite{BCSZ} again shows here that with our assumptions, $A_0\subset B_0$.  We must simply replace instances of $\partial A$ by $\partial A_{\call}\cup A_{\cals}$ in the paragraph spanning pp.~2405--2406 there and its sequel.  Pushing off $\partial B_0$ for each such component $B_0$, we will assume $A_{\call}$ is contained in the interior of $B$.

Now suppose $\acurve_0$ is a component of $A_{\cals}$ and let $B_0$ be the component of $B$ into which it is homotopic.  We again follow the proof of \cite[Lemma 4.1]{BCSZ}: fixing a base point in $B_0$, let $p\co\tilde{S} \to \mathrm{int}\, S$ be the cover corresponding to $\pi_1 B_0$.  The inclusion map $B_0\to S$ lifts to an embedding to a component $\tilde{B}_0$ of $p^{-1}(B_0)\subset\tilde{S}$, and because $\acurve_0$ is homotopic into $B_0$ it too lifts to a simple closed curve $\tilde{\acurve}_0$ in $\tilde{S}$.  Since $B_0$ is incompressible, the inclusion $\tilde{B}_0\to \tilde{S}$ induces an isomorphism at the level of $\pi_1$, and so each component of $X = \tilde{S} - \mathrm{int}\,\tilde{B}_0$ is homeomorphic to a half-open annulus.

If $\tilde{\acurve}_0$ meets $\partial\tilde{B}_0$, then the argument of the paragraph that spans pp. 2405--2406 in \cite{BCSZ} and its sequel again yields a contradiction to our minimality assumption (after the same adjustment as before).  Therefore $\tilde{\acurve}_0$ is disjoint from $\partial \tilde{B}_0$, and if $\tilde{\acurve}_0$ is not contained in $\tilde{B}_0$ then it is contained in an annular component $Z$ of $X$.  (Unlike in the proof of \cite[Lemma 4.1]{BCSZ} this can occur, since $\pi_1 \acurve_0 \cong \mathbb{Z}$.)

Since $\tilde{\acurve}_0$ is a homotopically nontrivial simple closed curve in $Z$, it cobounds an annulus with the component $\tilde{\bcurve}_0$ of $\partial\tilde{B}_0$ that bounds $Z$.  This annulus projects to a free homotopy between $\acurve_0$ and $\bcurve_0 = p(\tilde{\bcurve}_0)$, a component of $\partial B_0$.  Theorem 2.1 of \cite{Epstein} now implies that $\acurve_0$ is isotopic to $\bcurve_0$ and hence, pushing a bit further, into the interior of $B_0$.  This isotopy may be taken to be supported in a small enough neighborhood of the annulus bounded by $\acurve_0$ and $\bcurve_0$ that it leaves invariant all components of $A_{\call}\cup A_{\cals}$ inside $B_0$, and all components of $A_{\cals}$ outside the annulus.  After a finite sequence of such isotopies we have $A_{\call}\cup A_{\cals}\subset B$.

To complete the proof of (\ref{isotopic into}), fix a hyperbolic metric with convex boundary on $S$, and choose $\epsilon >0$ so that for each component $\mathfrak{a}$ of $A_{\cals}$, the following hold:\begin{enumerate}
\item  The $\epsilon$--neighborhood $\mathcal{N}_{\epsilon}(\mathfrak{a})$ is regular and contained in the component $A_0$ of $A$ containing $\mathfrak{a}$.
\item  Throughout the isotopy described above, $\mathcal{N}_{\epsilon}(\mathfrak{a})$ remains regular, and $\mathfrak{a}$ has distance at least $2\epsilon$ from every other component of $A_{\call} \cup A_{\cals}$.
\item   After the isotopy described above, $\mathcal{N}_{\epsilon}(\mathfrak{a}) \subset B$.  \end{enumerate}
By the first criterion above $A$ deformation retracts to the union of $A_{\call}$ with $\bigcup_{\mathfrak{a}} \mathcal{N}_{\epsilon}(\mathfrak{a})$ over the components $\mathfrak{a}$ of $A_{\cals}$.  By the second criterion, the isotopy of $A_{\call} \cup A_{\cals}$ extends to this union, and by the third, it takes it into $B$.  This establishes (\ref{isotopic into}).

We now turn to the proof of (\ref{homeo to iso}).  Using (\ref{isotopic into}), we will assume that $A \subset \mathrm{int}\, B$.  In particular, $A_{\call} \subset \mathrm{int}\, B_{\call}$.  Since $\pi_1$-injective maps preserve large parts, $B_{\call}$ is homeomorphic to a large subsurface of $A_{\call}$.  The last 3 paragraphs on \cite[p. 2406]{BCSZ} thus imply that each component of $\overline{B_{\call} - A_{\call}}$ is an annulus with exactly one boundary component in $A_{\call}$.  In particular, we note that $\chi(B_{\call}) = \chi(A_{\call})$, where $\chi(S)$ refers to the Euler characteristic of $S$.

Since $A$ is irredundant it follows that each annular component of $A$ is contained in an annular component of $B$, and that no two are contained in the same component.  Therefore $B_{\cals}$ has at least as many components as $A_{\cals}$.  If $B_{\cals}$ had more components than $A_{\cals}$, then the homeomorphic embedding $B \to A$ would either take two annular components into the same annular component of $A$, contradicting irredundancy of the image, or would take an annular component of $B$ into a component of $A_{\call}$.  But since the image of $B_{\call}$ is a large subsurface of $A_{\call}$ with the same Euler characteristic, each component of its complement is an annulus, and the latter possibility above again contradicts irredundancy of the image of $B$.

We thus find that each annular component of $B$ contains a unique component of $A$ as an incompressible sub-annulus.  Together with the assertions above regarding $A_{\call} \subset B_{\call}$, this implies (\ref{homeo to iso}).  

To establish (\ref{isotopic to equal}), we note that if $B$ is homotopic into $A$, then by (\ref{isotopic into}) it is isotopic to a subsurface of $A$.  This subsurface is necessarily irredundant, since $B$ is, hence the desired conclusion follows from (\ref{homeo to iso}).  \end{proof}

The following proposition extends \cite[Proposition 4.2]{BCSZ}.  Below we reference the ``large intersection'' $A \lcap B$ of large surfaces $A$ and $B$ from \cite[Definition 4.3]{BCSZ}.

\begin{prop}\label{essential intersection}  Suppose $A$ and $B$ are irredundant subsurfaces of an orientable compact surface $S$ with no $2$-sphere or torus components.  Then up to non-ambient isotopy there is a unique irredundant subsurface $C$ of $S$ with the following property: \begin{itemize}
\item[($\ast$)]  $C_{\call} = A_{\call} \lcap B_{\call}$, and for a polyhedron $K$ and a map $f\co K\to S$ such that $f_*(\pi_1 K_0)\neq 1$ for each component $K_0$ of $K$, $f$ is homotopic into each of $A$ and $B$ if and only if $f$ is homotopic into $C$.  \end{itemize}
Furthermore, there are subsurfaces $A' \subset S$ and $B' \subset S$, isotopic to $A$ and $B$, respectively, such that $\partial A'$ meets $\partial B'$ transversely and a union $C$ of components of $A' \cap B'$ satisfies $(\ast)$ above.  \end{prop}

\begin{dfn}\label{the real essential intersection}  If $A$ and $B$ are irredundant subsurfaces of an orientable compact surface $S$, we say an irredundant surface $C$ that satisfies condition ($\ast$) of Proposition \ref{essential intersection} \textit{represents the essential intersection $A\essint B$ of $A$ and $B$}.  \end{dfn}

Proposition \ref{essential intersection} implies in particular that each of $A$ and $B$ contains a subsurface that represents $A\essint B$, and that these subsurfaces are isotopic in $S$.  

\begin{proof}[Proof of Proposition \ref{essential intersection}]  We assume without loss of generality that $A,B \subset \mathrm{int}\, S$.  If $C$ and $C'$ are surfaces with property $(\ast)$, then $C$ is homotopic into $C'$ and vice-versa.  Hence Lemma \ref{homotopy to iso}(\ref{homeo to iso}) implies that they are isotopic, establishing uniqueness.

Now let $B_0$ be a representative of the isotopy class of $B$ in $S$ with the property that $\partial B_0$ meets $\partial A$ transversely in the smallest possible number of points, and let $C_0$ be the union of the components of $A\cap B_0$ that are large.  (In the language of \cite{BCSZ}, $C_0 = \call(A \cap B_0)$.)  The proof of \cite[Proposition 4.2]{BCSZ} implies that for any polyhedron $K$, every large map $f\co K \to S$ that is homotopic into $A$ and $B$ is also homotopic into $C_0$.  (Recall from Definition \ref{large'n'stuff} that $f\co K\to S$ is \textit{large} if $f_*(\pi_1 K_0)$ is non-abelian for each component $K_0$ of $K$.)  We will construct $C$ by adding annular components to $C_0$.

Suppose $K$ is a connected polyhedron and $f\co K\to S$ is a map with $f_*(\pi_1 K) \neq \{1\}$, homotopic into $A$ and $B$ but not $C_0$.  Let $A_1$ be a component of $A$ such that $f$ is homotopic into $A_1$, let $p\co \tilde{S}\to \mathrm{int}\, S$ be the covering space corresponding to $\pi_1 A_1$, and let $\tilde{A} \subset \tilde{S}$ be a component of $p^{-1}(A)$ mapping homeomorphically under $p$.  Since $A_1$ is $\pi_1$-injective in $S$, the inclusion-induced homomorphism $\tilde{A} \to \tilde{S}$ is an isomorphism and hence every component of $X = \tilde{S} - \mathrm{int}\, \tilde{A}$ is a half-open annulus.

Note that since $f$ is homotopic into $B$, it is homotopic into $B_0$.  Since $f$ is homotopic into $A_1$, it admits a lift $\tilde{f}$ to $\tilde{S}$; furthermore, the homotopy into $B_0$ lifts to a homotopy of $\tilde{f}$ to a map $g$ with image in $p^{-1}(B_0)$.  Let $\tilde{B}_0$ be the component of $p^{-1}(B_0)$ containing $g(K)$.  Unlike in the proof of \cite[Proposition 4.1]{BCSZ}, it is not necessarily true that $\tilde{B}_0$ intersects $\tilde{A}_1$.  We will treat the two cases separately.

Suppose first that $\tilde{B}_0\cap\tilde{A}_1\neq\emptyset$.  Then the argument that begins in the paragraph of \cite{BCSZ} spanning pp. 2407--2408 establishes that $\tilde{B}_0$, hence also $\tilde{f}$, deforms in $\tilde{S}$ into $\tilde{B}_0 \cap \tilde{A}_1$.  Projecting this homotopy of $\tilde{f}$ to $S$ yields a homotopy of $f$ into a component of $B_0\cap A_1$.  Since $f_*(\pi_1 K)\neq\{1\}$, this component is not a disk.  Since $f$ is not homotopic into $C_0$ this component is not large, so it is an annulus $Z_1$ which moreover is not parallel to any component of $C_0$. 

In this case let $C_1 = C_0\cup Z_1$, and let $A_1' = A$ and $B_1' = B_0$.  These are subsurfaces of $S$ respectively isotopic to $A$ and $B$, such that $C_1$ is a union of components of their intersection.
  
Suppose now that $\tilde{B}_0 \cap \tilde{A}_1 = \emptyset$, and let $Z$ be the component of $X$ containing $\tilde{B}_0$.  Since $\pi_1\tilde{B}_0$ contains $g_*(\pi_1 K)\neq \{1\}$, $\tilde{B}_0$ has a boundary component $\mathfrak{b}_0$ that is a homotopically non-trivial simple closed curve in $Z$.  Hence $\bcurve_0$ cobounds an annulus $Z_0 \subset Z$ together with $\mathfrak{a}_0 =  \partial Z$.  If any component of the frontier in $\tilde{S}$ of $p^{-1}(B_0)$ intersected $\mathfrak{a}_0$, there would thus be a disk in $Z_0$ with boundary $\alpha \cup \beta$, where $\alpha \subset \acurve_0$ and $\beta \subset \partial(p^{-1}(B_0))$.  If this did occur then $B_0$ could be isotoped to reduce the number of intersections with $A$, by the argument of the paragraph of \cite{BCSZ} spanning pp. 2405--2406.  Thus $\acurve_0 \cap p^{-1}(B_0) = \emptyset$.

Since $p$ projects $A_1$ homeomorphically, it sends $\mathfrak{a}_0$ homeomorphically to a component of $\partial A_1$ in $S$.  Since $\tilde{B}_0$ is a component of $p^{-1}(B_0)$, $p$ restricts on $\mathfrak{b}_0$ to a $k$-to-$1$ covering map to a component $\mathfrak{b}$ of $\partial B_0$ for some $k \geq 1$.  By the paragraph above, $\mathfrak{b}$ does not intersect $\mathfrak{a}$.  Furthermore, the annulus $Z_0$ bounded by $\mathfrak{a}_0$ and $\mathfrak{b}_0$ projects under $p$ to a free homotopy in $S$ between $\mathfrak{a}$ and the $k$th power of $\mathfrak{b}$.  Since $S$ is an orientable surface, by \cite[Lemma 2.4]{Epstein} $k=1$ and $\mathfrak{a}$ and $\mathfrak{b}$ bound an annulus $Z_1$ in $S$.  It still holds in this case that $g$, and hence also $f$, is homotopic into $Z_1$ since the annulus $Z\subset \tilde{S}$ containing $\tilde{B}_0$ deformation retracts to $\acurve_0$.

We may assume that $Z_1\cap A = \acurve$.  If this is not so then since $A$ is essential and irredundant, a component $A_2$ of $A$ intersects $\bcurve$.  If a component of $\partial A_2$ intersected $\bcurve$ then by an innermost disk argument there would be an isotopy of $B_0$ reducing the number of intersections of $\partial B_0$ with $\partial A$, a contradiction.  Therefore $\bcurve\subset A_2$, so $f$ is homotopic into $A_2$, and putting $A_2$ in the role of $A_1$ in the argument above we find that $\tilde{B}_0\cap\tilde{A}_2\neq \emptyset$; ie, we are in the first case.  So assuming that none of the possible choices of $A_1$ yields the first case, we have $Z_1\cap A = \acurve$.

A similar argument shows that $B_0\cap A_1 = \bcurve$, and it follows in this case that $A_1' = A\cup Z_1$ and $B_1'=B_0\cup Z_1$ are respectively isotopic to $A$ and $B$, and that $Z_1$ is a component of $A_1'\cap B_1'$.  Again in this case, let $C_1 = C_0 \cup Z_1$.

We now repeat the argument above but with $C_1$ in the role of $C_0$.  If there is a polyhedron $K$ and a map $f\co K\to S$ homotopic into $A$ and $B$ but not $C_1$, with $f_*(\pi_1 K)\neq\{1\}$, then this argument produces an essential annulus $Z_2\subset S$ with $f$ homotopic into $Z_2$.  We may take $Z_2$ either to be a component of $A\cap B$ or to intersect each in a distinct component of its frontier.  In either case $Z_2$ is disjoint from $C_1$, and there exist surfaces $A_2'$ and $B_2'$ respectively isotopic to $A$ and $B$ such that $C_2 = C_0\cup Z_1\cup Z_2$ is a union of components of $A_2'\cap B_2'$.

Iterating this process produces a sequence $\{C_n\}$ of subsurfaces of $S$ with the property that $C_n = C_{n-1}\cup Z_n$ for an essential annulus $Z_n$ disjoint from and not isotopic into $C_{n-1}$, which is either a component of $A\cap B$ or intersects $A$ and $B$ in distinct components of its frontier.  The process terminates at some finite $n$, since $A\cap B$ has only finitely many components, and each of $A$ and $B$ have only finitely many boundary components.  It then follows from the construction above that $C \doteq C_n$ has the property $(\ast)$.
\end{proof}

The result below, which we will use in the proof of Proposition \ref{carriers}, extends Proposition 4.4 of \cite{BCSZ}.  Its statement and proof follow those of its predecessor, but an additional case must be considered.

Below, for a subset $S$ of a topological space $X$, we refer to the \textit{frontier} of $S$ in $X$ as $\mathrm{fr}\, S \doteq \overline{S} \cap \overline{X-S}$.

\begin{prop}\label{pushing homotopy}  Suppose $B$ is an irredundant subsurface of a compact, orientable surface $S$ with no $2$-sphere components, and for a connected polyhedron $K$ let $f\co K \to B$ satisfy $f_*(\pi_1 K)\neq \{1\}$.  If $g\co K\to B$ is homotopic to $f$ in $S$, then: \begin{enumerate}
\item Either $f$ and $g$ are homotopic in $B$; or 
\item For distinct components $\acurve$ and $\bcurve$ of the frontier of $B$ that are parallel in $S$ but not $B$, $f$ is homotopic into $\acurve$, and $g$ into $\bcurve$, in $B$.  \end{enumerate}\end{prop}

\begin{remark}  To directly extend \cite[Proposition 4.4]{BCSZ} we must allow $K$ to be disconnected.  Such a result is obtained by applying Proposition \ref{pushing homotopy} component-by-component.\end{remark}

\begin{proof}  Assume $B \subset \mathrm{int}\, S$, and let $B_0$ be the component of $B$ containing $f(K)$.  Choosing a base point in $B_0$, we let $p\co \tilde{S} \to \mathrm{int}\, S$ be the cover corresponding to $\pi_1(B_0) < \pi_1(S)$.  By construction, the inclusion map $B_0 \to S$ lifts to an embedding to $\tilde{S}$ with image a subsurface which we denote by $\tilde{B}_0$, that carries the fundamental group of $\tilde{S}$.  Since $B_0$ is $\pi_1$-injective in $S$, each component of $\tilde{S} - \mathrm{int}\,\tilde{B}_0$ is homeomorphic to a half-open annulus.  In particular, there is a deformation retraction $r\co \tilde{S} \to \tilde{B}_0$.

Since $f$ maps $K$ into $B_0$, composing with the lift of the inclusion map gives a lift $\tilde{f}\co K \to \tilde{S}$ with $\tilde{f}(K) \subset \tilde{B}_0$; furthermore, the homotopy from $f$ to $g$ lifts to a homotopy $H$ from $\tilde{f}$ to a lift $\tilde{g}$ of $g$ with image in $p^{-1}(B)$.  If $\tilde{g}$ has image in $\widetilde{B}_0$, then $H_1 = p\circ r \circ H$ is a homotopy between $f$ and $g$ with image in $B_0$.

If $\tilde{g}$ does not map into $\tilde{B}_0$ then the component of $p^{-1}(B)$ containing its image lies in a component $Z$ of $\tilde{S}-\mathrm{int}\, B_0$, a half-open annulus.  In this case the time-$1$ map of $r\circ H$ has its image in the frontier $\tilde{\acurve} = Z\cap \tilde{B}_0$ of $Z$.  So $p\circ r\circ H$ is a homotopy of $f$ in $B$, into a component $\acurve = p(\tilde{\acurve})$ of the frontier of $B$.

Switching the roles of $f$ and $g$ in the argument above, we find that if $f$ and $g$ are not homotopic in $B$ then $g$ is homotopic in $B$ into a component $\bcurve$ of the frontier of $B$.  This is distinct from $\acurve$ and not parallel to it in $B$, since it follows from algebraic topology that two maps from a polyhedron (or more generally, a CW-complex) $K$ to $S^1$ that induce the same map on $\pi_1 K$ are homotopic.  See eg.~Exercise 2 of \cite[\S4.A]{Hatcher_AT}.  Let us now choose arcs from $\acurve$ and $\bcurve$ to the basepoint of $\pi_1 S$ and again denote by $\acurve$ and $\bcurve$ the elements of $\pi_1 S$ thus determined.  The non-trivial subgroup $f_*(\pi_1 K) = g_*(\pi_1 K)$ of $\pi_1 S$ is contained in both a conjugate of the subgroup $\langle \acurve\rangle$ generated by $\acurve$ and a conjugate of $\langle\bcurve\rangle$. 

Our hypotheses ensure that $\pi_1 S$ is isomorphic either to $\mathbb{Z}\oplus\mathbb{Z}$ or a Fuchsian group.  In either case, standard results ensure that any two cyclic subgroups with non-trivial intersection both lie in a single cyclic group.  (In the Fuchsian case see eg.~Theorems 2.3.3 and 2.3.5 of \cite{Katok}.)  Thus there exists $\gamma\in\pi_1 S$ such that certain conjugates of $\acurve$ and $\bcurve$ are powers of $\gamma$.  But since these conjugates represent the \textit{simple} closed curves $\acurve$ and $\bcurve$ they are primitive elements of $\pi_1 S$ (see eg.~\cite[Prop.~1.4]{FaMa}), and it follows that $\acurve$ is conjugate to $\bcurve^{\pm 1}$ in $\pi_1 S$.  Lemma 2.4 of \cite{Epstein} now implies that $\acurve$ and $\bcurve$ are parallel.
\end{proof}

The lemma below distills a fact from the proof of Proposition \ref{pushing homotopy} that we will use in the following section.

\begin{lemma}\label{homotopic out, homotopic in}  Let $B$ be a compact, connected incompressible subsurface of a surface $S$, and for a polyhedron $K$ suppose $f\co K\to B$ is homotopic into $S-B$.  Then $f$ is homotopic in $B$ into $\mathrm{fr}\, B$.  \end{lemma}

\begin{proof}  After pushing off of boundaries, we will assume that $B \subset \mathrm{int}\, S$ and $f$ maps into $\mathrm{int}\, B$.  Choose a base point for $\pi_1 S$ in $B$, and let $p\co\tilde{S}\to \mathrm{int}\,S$ be the cover corresponding to $\pi_1 B$.  If $\tilde{B}$ is the image in $\tilde{S}$ of the lift of the inclusion map $B\hookrightarrow S$ then every component of $\tilde{S}-\mathrm{int}\,\tilde{B}$ is homeomorphic to a half-open annulus, and there is a retraction $r\co \tilde{S}\to\tilde{B}$ that takes each such component to a component of $\partial \tilde{B}$.  The homotopy of $f$ out of $B$ lifts to a homotopy $\tilde{H}$ whose time-$0$ map has its image in $\tilde{B}$.  The time-$1$ map $\tilde{H}_1$ has its image in $\tilde{S}-\mathrm{int}\,\tilde{B}$, so $p\circ r\circ\tilde{H}$ is a homotopy of $f$ in $B$ to a map with its image in $\partial B$.\end{proof}

\section{Cylinders have bounded length}\label{shorter cylinders}

This section is dedicated to proving Theorem \ref{bdd length}:

\begin{thm}\label{bdd length}\BddLength\end{thm}

The proof uses the \textit{characteristic submanifold} of the manifold $X$ obtained by cutting $M$ along $S$, which has the key property that it captures all non-trivial homotopies in $X$.  We recall its definition below.

Say a $3$-manifold $X$ with boundary is \textit{simple} if:\begin{itemize}
\item $X$ is compact, connected, orientable, irreducible and boundary-irreducible;
\item no subgroup of $\pi_1(X)$ is isomorphic to $\mathbb{Z}\times\mathbb{Z}$; and
\item  $X$ is not a closed manifold with finite fundamental group.\end{itemize}
For a closed, orientable hyperbolic $3$-manifold $M$ containing an incompressible surface $S$, each component of the manifold obtained by cutting $M$ along $S$ is simple.

Below, an \textit{essential annulus} in a $3$-manifold $X$ with boundary is the image of an essential, non-degenerate homotopy $(S^1\times I,S^1\times\partial I)\to (M,S)$ (recall Definition \ref{basic stuff}) that is embedding.  If $P$ is an $I$-bundle over a surface $F$, we let $\partial_h P$ denote the associated $\partial I$-bundle, the \textit{horizontal} boundary of $P$, and denote by $\partial_v P$ (the \textit{vertical} boundary) the $I$-bundle over $\partial F$.

\begin{thm1}[Jaco--Shalen \cite{JaS}, Johansson \cite{Jo}]  Let $X$ be a simple $3$-manifold with nonempty boundary.  Up to ambient isotopy, its characteristic submanifold $\Omega$ is the unique compact
submanifold of $X$ with the following properties.
\begin{enumerate}
\item\label{thermometer} Every component of $\Omega$ is either an $I$-bundle $P$ over a
surface such that $P\cap\partial X=\partial_hP$, or a solid torus $S$ such that $S\cap\partial X$ is a collection of disjoint, embedded annuli in $\partial S$ that are homotopically non-trivial in $S$.
\item\label{sushi} Every component of the frontier of $\Omega$ is an
essential annulus in $X$.
\item\label{cheesequake} No component of $\Omega$ is ambiently isotopic in $X$ to a
submanifold of another component of $\Omega$.
\item\label{big ol' hash} If $\Omega_1$ is a compact
submanifold of $X$ such that (\ref{thermometer}) and (\ref{sushi}) hold with $\Omega_1$
in place of $\Omega$, then $\Omega_1$ is ambiently isotopic in $X$ to
a submanifold of $\Omega$.
\end{enumerate}
If $K$ is a polyhedron and $H\co (K \times I,K\times\partial I) \to (X,\partial X)$ is an essential, non-degenerate map, then $H$ is homotopic into $(\Omega,\Omega\cap\partial X)$.\end{thm1}

Let the \textit{characteristic set} of $X$ be $\Omega \cap \partial X$.  If $X$ is a component of the manifold obtained by cutting $M$ along $S$ then by the JSJ theorem its characteristic set carries a homotopic image of the time-$0$ map of any essential basic homotopy in $(M,S)$ (recall Definition \ref{BCSZ stuff}) that intersects $X$.  The first main result of this section identifies a sequence of subsurfaces that play a role analogous to the characteristic set for homotopies with length $k \geq 1$.  This extends Proposition 5.2.8 of \cite{BCSZ}.

Before we state the result, we translate Definition 5.1.1 of \cite{BCSZ} into our context.

\begin{dfn}\label{splitting}  A \textit{splitting surface} in a closed, orientable hyperbolic $3$-manifold $M$ is a transversely oriented, incompressible surface $S\subset M$ such that the manifold obtained by cutting $M$ along $S$ is a disjoint union of submanifolds $X^{\pm 1}$ with the property that $\caln_{\epsilon}\subset X^{\epsilon}$ for each $\epsilon\in\{\pm 1\}$, for $\caln_{\epsilon}$ as in Definition \ref{BCSZ stuff}.\end{dfn}

Separating, connected, two-sided incompressible surfaces are splitting surfaces, but note that $S$ is not required above to be connected.  In fact, given a non-separating connected, two-sided incompressible surface $S_0$ in $M$, the boundary $S$ of a regular neighborhood $\caln_0$ of $S_0$ becomes a splitting surface upon taking $X^{-1} = \caln_0$ and $X^{+1} = \overline{M-\caln_0}$, and giving each component of $S$ the transverse orientation pointing out of $X^{-1}$.

\begin{prop}\label{carriers}  Let $M$ be a closed, orientable hyperbolic $3$-manifold and $S\subset M$ a splitting surface, and decompose the manifold obtained by cutting $M$ along $S$ into submanifolds  $X^{\pm 1}$ as in Definition \ref{splitting}.  For each $\epsilon \in \{\pm 1\}$ there is a sequence of essential (possibly empty) subsurfaces $(\Psi_k^{\epsilon})_{k \in \mathbb{N}}$ of $S$, such that $\Psi_1^{\epsilon} \subset \Omega^{\epsilon} \cap \partial X^{\epsilon}$, where $\Omega^{\epsilon}$ is the characteristic submanifold of $X^{\epsilon}$, and for each $k \in \mathbb{N}$ we have:  \begin{enumerate}
\item\label{containment}  $\Psi_k^{\epsilon} \supset \Psi_{k+1}^{\epsilon}$; and 
\item\label{eats homotopies}  If $K$ is a polyhedron with $\pi_1 K\neq \{1\}$ and $H\co K\times I \to M$ is a reduced homotopy in $(M,S)$ of length $k$, starting on the $\epsilon$ side, then $H_0$ is homotopic in $S$ to a map with image in $\Psi_k^{\epsilon}$.  Conversely, for such a polyhedron $K$ if $f\co K\to S$ is $\pi_1$-injective and homotopic into $\Psi_k^{\epsilon}$ then there exists such a homotopy $H$ with $H_0=f$; and
\item\label{large part} $(\Psi_k^{\epsilon})_{\call} = \Phi_k^{\epsilon}$, where $\Phi_k^{\epsilon}$ is the surface identified in \cite[Proposition 5.2.8]{BCSZ}.  \end{enumerate}
A surface with the properties above is determined up to isotopy in $S$ by the requirement that it be irredundant.  \end{prop}

Below we will briefly review some of definitions and results proved in Section 5 of \cite{BCSZ}. These were proven there under the hypothesis that $M$ is a \textit{knot manifold}, with a single torus boundary component, whereas we take $M$ closed.  However, they depend only on the results on large intersection developed in \cite[\S 4]{BCSZ} and basic facts about $I$-bundles, and so carry over to our context without alteration.  The blanket hypotheses below are those of Proposition \ref{carriers}; in each case we paraphrase the result or definition from \cite{BCSZ} that is referenced.\begin{description}
  \item[5.2.1]  Let $(\Sigma^{\epsilon},\Phi^{\epsilon})$ be the $(I,\partial I)$-bundle pair that is the union of all $I$-bundle components of the characteristic submanifold of $X^{\epsilon}$.  
  \item[Proposition 5.2.8]  There is a sequence $\{\Phi_1^{\epsilon}\supset\Phi_2^{\epsilon}\supset\hdots\}$ of large subsurfaces of $\Phi^{\epsilon}$, with $\Phi_1^{\epsilon} = (\Phi^{\epsilon})_{\call}$, that satisfies property (\ref{eats homotopies}) of Proposition \ref{carriers} with the hypothesis that $\pi_1 K\neq\{1\}$ replaced by the assertion that $H_0$ is large.  The $\Phi_i^{\epsilon}$ are determined up to isotopy by this property.
  \item[Proposition 5.3.1]  There is a homeomorphism $h_k^{\epsilon} \co \Phi_k^{\epsilon} \to \Phi_k^{(-1)^{k+1}\epsilon}$, for each $k\in\mathbb{N}$, such that if $H \co K\times I \to M$ is a reduced homotopy of length $k$ starting on the $\epsilon$-side with large time-$0$ map, then there exists $f\co K \to \Phi_k^{\epsilon}$ such that $H_0$ is homotopic to $f$ and $H_1$ to $h_k^{\epsilon} \circ f$.
\end{description}
To motivate the existence of $h_k^{\epsilon}$, we note that the analog of Proposition \ref{carriers}(\ref{eats homotopies}) implies the inclusion of $\Phi_k^{\epsilon}$ is the time-$0$ map of a length-$k$ homotopy $H$ with target $(M,S)$.  Since $H$ is length-$k$ the image of $H_1$ lies in $\partial X^{(-1)^{k+1}\epsilon}$, and since $H$ can be run backwards this image is homotopic into $\Phi_k^{(-1)^{k+1}\epsilon}$.

The precise definition of the $h_k^{\epsilon}$ is as follows.  Let $\tau_{\epsilon}$ be the fixed-point free involution of $\Phi^{\epsilon}$ that exchanges the endpoints of $I$-fibers.  Then $h_1^{\epsilon}$ is defined to be the restriction of $\tau_{\epsilon}$ to $\Phi_1^{\epsilon}$.  For $k>1$, $h_{k}^{\epsilon}$ is defined recursively by composing $\tau_{\pm \epsilon}$ with a homotope of the restriction of $h_{k-1}^{\epsilon}$.  In \cite{BCSZ} it is proven:\begin{description}
\item[Proposition 5.3.4]  For $k>1$, $h_{k-1}^{\epsilon}|_{\Phi_k^{\epsilon}}$ is homotopic in $S$ to an embedding $g_{k-1}^{\epsilon}\co \Phi_k^{\epsilon} \to \Phi_1^{(-1)^{k-1}\epsilon}$ such that $h_k^{\epsilon}$ is homotopic in $S$ to $\tau_{(-1)^{k-1}\epsilon}\circ g_{k-1}^{\epsilon}$.
\item[Proposition 5.3.5] $h_{k-1}^{\epsilon}(\Phi_k^{\epsilon})$ is isotopic in $S$ to $\left(\Phi_{k-1}^{(-1)^k\epsilon} \essint \Phi_1^{(-1)^{k-1}\epsilon}\right)_{\call}$.\end{description}
The statements above are special cases of the results cited.  Our phrasing of the latter implicitly uses our Proposition \ref{essential intersection} (also see above it, and Definition \ref{the real essential intersection}).

The lemma below is a version of \cite[Lemma 5.2.4]{BCSZ}, where the original hypothesis that the homotopy $H$ in question has large time-$0$ map has been replaced here by the assertion that $H$ maps into $\Sigma^{\epsilon}$.  In the original version this follows from the largeness hypothesis; the remainder of its proof carries through without revision.

The \textit{standard} essential basic homotopy referenced below is from Definition 5.2.3 of \cite{BCSZ}, which in turn refers to the \textit{fundamental} homotopy defined in 5.2.1 there.  For a component $P$ of $\Sigma^{\epsilon}$, which is an $I$-bundle in $X^{\epsilon}$ such that $P\cap\partial X^{\epsilon}$ is the associated $\partial I$-bundle (see above), the fundamental homotopy has domain $P\cap\partial X^{\epsilon}$ and takes $I$-fibers to $I$-fibers.
 
\begin{lemma}\label{vertical}  For $\epsilon \in \{\pm 1\}$ and a polyhedron $K$, let $H \co (K \times I,K\times\partial I) \to (\Sigma^{\epsilon},\Phi^{\epsilon})$ be an essential basic homotopy.  Then $H$ is homotopic as a map of pairs to a standard essential basic homotopy.  In particular, $H_1$ is homotopic in $P$ to $\tau_{\epsilon} \circ H_0$.  \end{lemma}

The lemma below extends the conclusion of \cite[Proposition 5.3.1]{BCSZ} to certain reduced homotopies whose time-$0$ maps are not necessarily large.

\begin{lemma}\label{stronger uniqueness} For $\epsilon \in \{\pm 1\}$ and $k \in \mathbb{N}$, suppose $H$ is a reduced homotopy in $(M,S)$ of length $k$ that starts on the $\epsilon$ side, with domain a polyhedron $K$, such that $H_0$ is homotopic in $S$ into $\Phi_k^{\epsilon}$ but not into $\partial\Phi_k^{\epsilon}$.  Then $H_1$ is homotopic in $S$ to $h_k^{\epsilon} \circ f$, where $H_0$ is homotopic to $f\co K\to\Phi_k^{\epsilon}$.  \end{lemma}

\begin{proof}  We will assume without loss of generality that $K$ is connected, since the desired homotopy can be constructed component-by-component.  We prove the result first for $k=1$; thus assume that $H\co (K\times I,K\times \partial I) \to (X^{\epsilon},\partial X^{\epsilon})$ is an essential basic homotopy.  Applying the JSJ theorem, after homotoping $H$ through maps of pairs to $(X^{\epsilon},\partial X^{\epsilon})$ we will assume that it maps into $\Omega^{\epsilon}$.

Let $P$ be an $I$-bundle component of $\Omega^{\epsilon}$ such that $P\cap\partial X^{\epsilon}\subset\Phi_1^{\epsilon}$ contains the image of a map $f$ homotopic to $H_0$.  Since $f$ is not homotopic into $\partial \Phi_1^{\epsilon}$, Lemma \ref{homotopic out, homotopic in} implies that $f$ is not homotopic out of $P\cap\partial X^{\epsilon}$, so $H_0$ and hence all of $H$ maps into $P$.  Lemma \ref{vertical} now yields the conclusion in this case, since $h_1^{\epsilon} = \tau_{\epsilon}|_{\Phi_1^{\epsilon}}$.

Now take $k>1$ and assume that the lemma holds for all reduced homotopies of length $k-1$.  Given a reduced homotopy $H$ of length $k$ that satisfies the hypotheses, writing $H$ as the composition of essential basic homotopies $H^1,\hdots,H^k$, we have that the composition of $H^1,\hdots,H^{k-1}$ has time-$1$ map homotopic to $h_{k-1}^{\epsilon}\circ f$, where $H_0$ is homotopic to $f\co K\to\Phi_k^{\epsilon}\subset\Phi_{k-1}^{\epsilon}$.

Let $g_{k-1}^{\epsilon} \co \Phi_k^{\epsilon} \to \Phi_1^{(-1)^{k-1}\epsilon}$ be the embedding supplied by \cite[Proposition 5.3.4]{BCSZ}, homotopic to the restriction of $h_{k-1}^{\epsilon}$ and so that $h_k^{\epsilon}$ is homotopic to $\tau_{(-1)^{k-1}\epsilon}\circ g_{k-1}^{\epsilon}$.  Let $P$ be the $I$-bundle component of $\Omega^{(-1)^{k-1}\epsilon}$ such that $g_{k-1}^{\epsilon}\circ f$ maps into $\partial_h P$.  Since $f$ is not homotopic into $\partial \Phi_k^{\epsilon}$, the same holds true for $g_{k-1}^{\epsilon}\circ f$ in $\partial_h P$.

Since $H^k_0 = H^{k-1}_1$ it is homotopic in $S$ to $g_{k-1}^{\epsilon}\circ f$.  It thus follows from the JSJ theorem as in the $k=1$ case that $H^k$ is homotopic as a map of $(I,\partial I)$-bundle pairs into $P$, and furthermore by Lemma \ref{vertical} that $H^k_1$ is homotopic to $\tau_{(-1)^{k-1}\epsilon}\circ g_{k-1}^{\epsilon} \circ f$.  Therefore $H^k_1 = H_1$ is homotopic to $h_k\circ f$, and the lemma follows by induction.
\end{proof}

Because solid torus components of $\Omega$ may have many components of intersection with $\partial X$, no homeomorphism analogous to $h_k^{\epsilon}$ is uniquely defined on $\Psi_k^{\epsilon}$.  But it is still true that every reduced homotopy is tracked by a homotopy of a surface containing the image of its time-$0$ map.

\begin{lemma}\label{along for the ride}  For $\epsilon\in\{\pm 1\}$ and $k\in\mathbb{N}$, suppose $H$ is a reduced homotopy in $(M,S)$ of length $k$ that starts on the $\epsilon$ side, with domain a connected, non-simply connected polyhedron $K$, such that $H_0$ is homotopic in $S$ to a map $f$ with image in an annulus $A \subset (\Omega^{\epsilon} \cap \partial X^{\epsilon})$. There is a reduced homotopy $J$ in $(M,S)$ of length $k$ that starts on the $\epsilon$ side, with domain $A$, such that $H_1$ is homotopic to $J_1\circ f$.\end{lemma}

\begin{proof}  Consider the case that $H \co (K \times I,K\times\partial I) \to (X^{\epsilon}, \partial X^{\epsilon})$ is an essential basic homotopy, for $\epsilon \in \{\pm 1\}$.  By the JSJ theorem, after a  homotopy through maps $(K\times I,K\times \partial I)\to(X^{\epsilon},\partial X^{\epsilon})$ we may assume $H$ maps into some component $P$ of the characteristic submanifold $\Omega^{\epsilon}$.

If the annulus $A$ supplied by the hypotheses does not lie in $P$ then Lemma \ref{homotopic out, homotopic in} implies that $H_0$ is homotopic into a component $\bcurve$ of $\partial (P\cap\partial X^{\epsilon})$.  The subgroups of $\pi_1S$ respectively generated by $\bcurve$ and the core circle $\acurve$ of $A$ thus share the non-trivial subgroup $(H_0)_*(\pi_1 K)$.  Since $S$ is orientable and $\acurve$ and $\bcurve$ are simple this implies they generate identical subgroups, so $\acurve$ and $\bcurve$ are parallel in $S$ by \cite[Lemma 1.4]{Epstein}.  It follows that even if $A$ does not lie in $P\cap \partial X^{\epsilon}$ it is still isotopic into it in $\partial X^{\epsilon}$.

If $P$ is an $I$-bundle component of $\Omega^{\epsilon}$ then applying Lemma \ref{vertical}, after a further homotopy we may assume that $H$ is standard.  If $A$ lies outside of $P$ then by the paragraph above we may homotope $H_0$ so that its image lies in an annular neighborhood $B\subset P\cap\partial X^{\epsilon}$ of $\bcurve$, isotopic to $A$, with the property that $B$ is a component of $\pi^{-1}(\pi(B))\cap\partial X^{\epsilon}$.  (Here $\pi$ is the bundle projection of $P$.)  Since $H$ is standard this determines a homotopy of $H$ to a standard homotopy in the restriction of $\pi$ to $\pi^{-1}(\pi(B))$.

A homotopy of $A$ through $X^{\epsilon}$ is now determined by composing the isotopy $J^0$ from $A$ to $B$ with the restriction $J^1$ to $B$ of the fundamental homotopy of $P\cap\partial X^{\epsilon}$.  This becomes a basic essential homotopy $J$ upon pushing $J^0.J^1\co A\times I\to X^{\epsilon}$ off of $\partial X^{\epsilon}$ on $\mathrm{int}\,I$.  Since $f$ is homotopic to $H_0$ in $\partial X^{\epsilon}$, Proposition \ref{pushing homotopy} now implies that $(J^0)_1\circ f$ is homotopic to $H_0$ in $B$: possibility (2) there does not occur since the components of $\partial B$ are not parallel in $S$, which has genus at least two.  Since $H$ and $J^1$ are standard, it now follows that $J_1\circ f$ is homotopic to $H_1$.

Now suppose $P$ is a solid torus component of $\Omega^{\epsilon}$, and let $B$ and $C$ be the components of $P\cap\partial X^{\epsilon}$ containing the images of $H_0$ and $H_1$, respectively.  As in the previous case, if $A\neq B$ then there is an isotopy $J^0$ from $A$ to $B$, and $H_0$ and $(J^0)_1\circ f$ are homotopic in $B$.  We now require a homotopy $J_1$ from $B$ to $C$ to replace the fundamental homotopy of the previous case.  We construct this below.

Fix a homeomorphic lift $\tilde{B}$ of $B$ to the cover $p\co \tilde{P}\to P$ corresponding to $\pi_1 B$, let $\tilde{H}$ be the lift of $H$ with $\tilde{H}_0(K)\subset\tilde{B}$, and let $\tilde{C}$ be the component of $p^{-1}(C)$ containing the image of $\tilde{H}_1$.  Note that since $C$ is parallel to $B$ on $\partial P$, $\tilde{C}$ is also homeomorphic lift of $C$.  Moreover, since $\tilde{B}$ and $\tilde{C}$ carry $\pi_1 \tilde{P}$ there is a product structure on $\tilde{P}$, $\tilde{P}\cong X\times I$ for an annulus $X$, with $\tilde{B}\cong X\times\{0\}$ and $\tilde{C}\cong X\times\{1\}$.  Restricting the fundamental homotopy of this product structure to $\tilde{B}$ yields a homotopy $\tilde{J}^1\co\tilde{B}\times I \to \tilde{P}$ such that $(\tilde{J}^1)_1$ is a homeomorphism to $\tilde{C}$.

Let $J^1$ be $p\circ\tilde{J}^1$ following the lift $B\to\tilde{B}$ of the inclusion map $B\hookrightarrow P$.  Now define a homotopy $J$ through $X^{\epsilon}$ with domain $A$ by pushing the composition $J^0.J^1$ off of $\partial X^{\epsilon}$ on $\mathrm{int}\,I$.  Lemma \ref{vertical} implies that $\tilde{H}$ is homotopic to a standard homotopy with respect to the product structure on $\tilde{P}$, so since $(J^0)_1\circ f$ is homotopic in $B$ to $H_0$ it follows that $J_1\circ f$ is homotopic in $C$ to $H_1$.

This completes the proof of the essential basic case.  The lemma now follows from this case and induction on the length $k$ of the reduced homotopy.\end{proof}

\begin{proof}[Proof of Proposition \ref{carriers}]  We will prove the proposition by induction.  Let $\Psi_1^{\pm 1}$ be obtained from $\Omega^{\pm 1} \cap \partial X^{\pm 1}$ by discarding redundant annuli, where $\Omega^{\pm 1}$ is the characteristic submanifold of $X^{\pm 1}$.  Property (\ref{eats homotopies}) for $\Psi_1^{\pm 1}$ holds by the enclosing property of the JSJ theorem, and we note that $(\Psi_1^{\pm 1})_{\call} = \Phi_1^{\pm 1}$.

Now let $m\geq 2$ be given, and suppose that for each $\epsilon \in \{\pm 1\}$ we have identified a sequence of subsurfaces
$$ \Psi_1^{\epsilon} \supset \Psi_2^{\epsilon} \hdots \supset \Psi_{m-1}^{\epsilon}, $$
such that for each $k < m$, $\Psi_k$ satisfies (\ref{eats homotopies}) and $(\Psi_k^{\epsilon})_{\call} = \Phi_k^{\epsilon}$. We will further assume (after discarding some annuli if necessary) that $\Psi_k^{\epsilon}$ is irredundant for $k<m$.

Before we define $\Psi_m^{\epsilon}$, we let $P_m^{\epsilon}$ be a subsurface of $\Phi_{m-1}^{(-1)^{m}\epsilon}$ representing: 
$$\Phi_{m-1}^{(-1)^{m}\epsilon} \essint \Psi_1^{(-1)^{m+1}\epsilon}$$ 
By Proposition \ref{essential intersection}, $(P_m^{\epsilon})_{\call}$ is maximal among large surfaces of $\Phi_{m-1}^{(-1)^m\epsilon}$ that admit a homotopy of length one starting on the $(-1)^{m+1}\epsilon$-side.  If a large subsurface $A$ of $\Phi_{m-1}^{(-1)^m\epsilon}$ admits an essential homotopy of length one starting on the $(-1)^{m+1}\epsilon$-side, then $(h_{m-1}^{\epsilon})^{-1}(A)$ admits a homotopy of length $m$ starting on the $\epsilon$-side; thus \cite[Proposition 5.2.8]{BCSZ} implies that $h_{m-1}^{\epsilon}(\Phi_m^{\epsilon})$ has the same maximality property as $(P_m^{\epsilon})_{\call}$.  Therefore by Proposition \ref{homotopy to iso}(\ref{isotopic to equal}), these are isotopic subsurfaces of $\Phi_{m-1}^{(-1)^m\epsilon}$.

We now define $\Psi_m^{\epsilon} = \Phi_m^{\epsilon} \cup \left(\bigcup A_i \right)\cup \left(\bigcup B_j\right) \cup \left( \bigcup C_k \right)$, where the $A_i$, $B_j$, and $C_k$ are annuli defined as follows:  \begin{enumerate}
\item  Let $\{A_i\}$ be the set of annular components of $\Psi_{m-1}^{\epsilon}$ that admit a reduced homotopy of length $m$.
\item  Let $\{\mathfrak{b}_j\}$ be the set of components of the frontier in $S$ of $\Phi_{m-1}^{\epsilon}$ such that $\mathfrak{b}_j$ is not isotopic into $\Phi_m^{\epsilon}$ but $\mathfrak{b}_j$ admits a reduced homotopy of length $m$, and for each $j$ let $B_j$ be a regular neighborhood of $\mathfrak{b}_j$ in $\Phi_{m-1}^{\epsilon} - \mathrm{int}\,\Phi_m^{\epsilon}$.  
\item  Let $\{C'_k\}$ be the set of annular components of $P_m^{\epsilon}$ that are not boundary parallel in $\Phi_{m-1}^{(-1)^{m}\epsilon}$.  For each $k$, let $C_k$ be an annulus isotopic in $\Phi_{m-1}^{\epsilon}$ to $(h_{m-1}^{\epsilon})^{-1}(C'_k)$ and disjoint from $\Phi_m^{\epsilon} \cup \bigcup B_j$.
\end{enumerate}

Properties (\ref{containment}) and (\ref{large part}) are clear from this construction.  Since $\Psi_m^{\epsilon}$ admits a reduced homotopy of length $m$ by construction, it remains only to show for a reduced homotopy $H\co (K \times I,K\times\partial I)\to (M,S)$ of length $m$ that $H_0$ is homotopic into $\Psi_m^{\epsilon}$.

Write $H$ as a composition of essential basic homotopies $H^1,\hdots,H^m$.  Since the composition $H^1.H^2\hdots H^{m-1}$ is a reduced homotopy of length $m-1$, by hypothesis $H_0 = (H^1)_0$ is homotopic into $\Psi_{m-1}^{\epsilon}$.  If $H_0$ is homotopic into an annular component of $\Psi_{m-1}^{\epsilon}$, then by Lemma \ref{along for the ride}, this component admits a reduced homotopy of length $k$; hence it is of the form $A_i$ for some $i$.  We thus assume below that this does not hold, hence that $H_0$ is homotopic into $\Phi_{m-1}^{\epsilon}$.

If $H_0$ is homotopic into $\Phi_m^{\epsilon}$ then we are done, so let us assume this is not the case.  In particular, by Proposition 5.2.8 of \cite{BCSZ}, $H_0$ is not large. If $H_0$ is homotopic into a boundary curve of $\Phi_{m-1}^{\epsilon}$ that is not homotopic into $\Phi_m^{\epsilon}$, then by Lemma \ref{along for the ride} again, the corresponding boundary component is of the form $\mathfrak{b}_j$ for some $j$.

By the paragraph above we may assume that $H_0$ is homotopic into $\Phi_{m-1}^{\epsilon}$ but not into $\partial \Phi_{m-1}^{\epsilon}$.  Lemma \ref{stronger uniqueness} therefore implies that $(H^{m-1})_1$ is homotopic to $h_{m-1}^{\epsilon}\circ f \subset \Phi_{m-1}^{(-1)^{m}\epsilon}$ in $S$, where $f\co K\to\Phi_{m-1}^{\epsilon}$ is homotopic to $H_0$.  It follows that $h_{m-1}^{\epsilon}\circ f$ admits an essential homotopy of length one, hence by Proposition \ref{essential intersection} it is homotopic into a component $C'$ of $P_m^{\epsilon} \subset \Phi_{m-1}^{\epsilon}$.

If $h_{m-1}^{\epsilon}\circ f$ were not homotopic into $C'$ in $\Phi_{m-1}^{(-1)^m\epsilon}$, then Proposition \ref{pushing homotopy} would imply in particular that it is homotopic in $\Phi_{m-1}^{(-1)^m\epsilon}$ into a boundary component.  But then $f$, and hence $H_0$ would be homotopic to a boundary component of $\Phi_{m-1}^{\epsilon}$, contradicting our assumption.  Hence $h_{m-1}^{\epsilon}\circ f$ is homotopic into $C'$ in $\Phi_{m-1}^{(-1)^m\epsilon}$.

Recalling from above that $(P_m^{\epsilon})_{\call}$ is isotopic in $\Phi_{m-1}^{(-1)^m\epsilon}$ to $h_{m-1}^{\epsilon}(\Phi_m^{\epsilon})$, we find that $C'$ is an annulus since we have assumed $H_0$ is not homotopic into $\Phi_m^{\epsilon}$.  Therefore $C'$ is of the form $C'_k$ for some $k$, and we are in case (3) above.
\end{proof}

The second main result of this section asserts that the sequence $\{\Psi_k^{\epsilon}\}$ is shrinking.  We cannot hope to establish that $\Psi_k^{\epsilon}$ is properly larger than $\Psi_{k+1}^{\epsilon}$ for each $k$.  Indeed, in the case of interest to us (when $S$ is the boundary of a regular neighborhood of a non-separating surface) $\Psi_k^{\epsilon}$ is identical to $\Psi_{k+1}^{\epsilon}$ for each odd or even $k$ (depending on $\epsilon$).  Instead we obtain the following extension of \cite[Proposition 5.3.9]{BCSZ}.

\begin{prop}\label{shrinkage}  Let $M$ be a closed, orientable hyperbolic $3$-manifold and $S\subset M$ a splitting surface that is not a fiber or a semi-fiber, and decompose the manifold obtained by cutting $M$ along $S$ into submanifolds $X^{\pm 1}$ as in Definition \ref{splitting}.  For $\epsilon \in\{\pm 1\}$, let $\Psi_1^{\epsilon} \supset \Psi_2^{\epsilon} \supset \cdots$ be a sequence of irredundant surfaces that satisfy Theorem \ref{carriers}.  Then for each $k$, $\Psi_k^{\epsilon}$ is not homotopic into $\Psi_{k+2}^{\epsilon}$.  \end{prop}

\begin{proof}  Proposition 5.3.9 of \cite{BCSZ} asserts that in this situation $\Phi_k^{\epsilon}$ is not homotopic into $\Phi_{k+2}^{\epsilon}$ for any $k \in \mathbb{N}$ or $\epsilon \in \{\pm 1\}$, so the result holds as long as $\Psi_k^{\epsilon}$ has nonempty large part.  Suppose therefore that for some $k$, $\Psi_k^{\epsilon}$ is a disjoint union of annuli homotopic into $\Psi_{k+2}^{\epsilon}$.

Let $H$ be a reduced homotopy in $(M,S)$ of length $k+2$ with domain $\Psi_{k+2}$ that starts on the $\epsilon$-side, and write $H$ as the composition of $H'', H'$, each starting on the $\epsilon$-side, where $H'$ has length $2$ and $H''$ length $k$.  Since $H'_1(\Psi_{k+2}^{\epsilon})$ admits a reduced homotopy of length $k$, Proposition \ref{carriers} implies  that $H'_1$ is homotopic to a map $f \co \Psi_{k+2}^{\epsilon}\to\Psi_k^{\epsilon}$.  After applying the homotopy that takes $\Psi_k^{\epsilon}$ into $\Psi_{k+2}^{\epsilon}$, we may take $f$ to map into $\Psi_{k+2}^{\epsilon}$.  It follows that there exists a homotopy of length $2$ in $(M,S)$ with domain and target $\Psi_{k+2}^{\epsilon}$.

Associate a directed graph $G$ to this homotopy as follows:  $G$ has a vertex $v$ for each component of $\Psi_{k+2}^{\epsilon}$, and a directed edge joining $v$ to $v'$ if and only if the component associated to $v$ is taken to the component associated to $v'$ by the time-$1$ map of the homotopy described above.  Then every vertex has a unique edge that leaves it, and so $G$ has a cycle.

We associate to a cycle $v_0,\hdots,v_{m-1}$ a map of a torus into $(M,S)$ as follows.  For $0\leq i< m$, let $\acurve_i$ be the core of the component of $\Psi_{k+2}^{\epsilon}$ corresponding to $v_i$, and let $F^i\co (S^1\times I,S^1\times\partial I) \to (M,S)$ be a reduced homotopy of length $2$ with $F^i_0 = \acurve_i$ and $F^i_1=\acurve_{i+1}$ (where $i+1$ is taken modulo $m$).  Dividing a torus $T$ into $m$ concentric essential annuli $A_i$, each homeomorphic to $S^1 \times I$, we obtain a map $F\co T \to M$ that restricts on $A_i$ to $F^i$ for each $i$.  Since each $F^i$ is essential, $F$ is essential, contradicting hyperbolicity of $M$.
\end{proof}

We may now prove Theorem \ref{bdd length}, which extends Theorem 5.4.1 of \cite{BCSZ}.  

\begin{proof}[Proof of Theorem \ref{bdd length}]  If $S$ is non-separating, we replace $S$ by the boundary $\tilde{S}$ of a regular neighborhood, yielding a separating surface with two components of genus $g$.  If $S$ is separating we take $\tilde{S} =S$, and in either case let $X^{\pm 1}$ be the components of the manifold obtained by cutting $M$ along $S$.    For $\epsilon \in \{\pm 1\}$, let $\Psi_1^{\epsilon} \supset \Psi_2^{\epsilon} \supset \hdots$ be a sequence of irredundant surfaces that satisfies the conclusion of Proposition \ref{carriers}.  

We now briefly review the proof of Theorem 5.4.1 of \cite{BCSZ}.  Given a large surface $A$, a complexity of $A$ is defined as $c(A) = g(A)-3\chi(A)/2-|A|$, where $\chi(A)$ is the Euler characteristic of $A$, $|A|$ is the number of its components, and $g(A)$ is the sum of their genera.  It is easy to see that if $A$ is nonempty and large, then $c(A) >0$.  The key fact established in the proof of Theorem 5.4.1 is that if $A$ and $B\subset A$ are large surfaces with even Euler characteristic, then $c(B) <c(A)$ unless $A$ is a regular neighborhood of $B$.

Fixing $\epsilon \in \{\pm 1\}$, consider the subsequence 
$$\Phi_1^{\epsilon} \supset \Phi_3^{\epsilon} \supset \hdots$$ 
This is strictly decreasing by \cite[Proposition 5.3.9]{BCSZ}, and consists of large surfaces with even Euler characteristic by \cite[Corollary 5.3.8]{BCSZ}.  Thus for each $i \geq 0$, $c(\Phi_{2i+1}^{\epsilon}) > c(\Phi_{2i+3}^{\epsilon})$.  If $S$ is separating, then $c(\tilde{S}) = c(S) = 4g-4$, and otherwise $c(\tilde{S}) = 8g-8$.   Taking $m_{S} = 4g-4$ in the separating case and $m_S = 8g-8$ in the non-separating case, it follows that for $i > m_S$, $\Phi_{2i+1} = \emptyset$.  

The discussion above is enough to establish \cite[Theorem 5.4.1]{BCSZ}.  In our situation of interest, it establishes that $\Psi_{2i+1}^{\epsilon}$ is a disjoint union of annuli for $i>m_S$.  Since $\Psi_i^{\epsilon}$ is irredundant, $\Psi_{2m_S+3}$ has at most $3g-3$ components in the separating case, and $6g-6$ otherwise.  (This uses the standard fact that a collection of disjoint, non-parallel, essential simple closed curves on a closed surface of genus $g$ has at most $3g-3$ members.)  Since Proposition \ref{shrinkage} implies $\Psi_{2i+1}^{\epsilon}$ is not homotopic into $\Psi_{2i+3}^{\epsilon}$, if these are unions of irredundant collections of annuli then $\Psi_{2i+3}^{\epsilon}$ has fewer components than $\Psi_{2i+1}^{\epsilon}$.  Thus taking $n_S = 3g-3$ in the separating case and $n_S = 6g-6$ otherwise, we find that $\Psi_{2i+1} = \emptyset$ for $i > m_S+n_S$.

By Proposition \ref{carriers}, the time-$0$ map of a reduced homotopy in $(M,\tilde{S})$ with length $k$ that starts on the $\epsilon$-side is homotopic into $\Psi_k^{\epsilon}$.  Therefore $k \leq 2(m_S+n_S)+2$.  If $S$ is separating, we therefore find that homotopies in $(M,S) = (M,\tilde{S})$ have length at most $14g-12$.  If $S$ is non-separating, a reduced homotopy of length $k$ in $(M,S)$ determines a reduced homotopy of length $2k-1$ in $(M,\tilde{S})$.  Thus in this case we have for a homotopy of length $k$ in $(M,S)$ that $2k-1 \leq 2(14g-14)+2$, so $k \leq 14g-13$.  The theorem follows.  \end{proof}

\bibliographystyle{plain}
\bibliography{posgrad_bib2}
\end{document}